\documentclass[a4paper,11pt,reqno]{amsart}
\usepackage[centering,includeheadfoot,margin=2.5 cm]{geometry}
\usepackage{graphics,enumitem,epsfig,textcomp}
\usepackage{amsfonts,amsmath,amssymb,euscript,color,mathrsfs}
\usepackage[utf8]{inputenc}	

\usepackage{cite}

\usepackage{hyperref}
\hypersetup{
	colorlinks,
	citecolor=green,
	filecolor=black,
	linkcolor=blue,
	urlcolor=blue
}
\hypersetup{linktocpage}

\usepackage[foot]{amsaddr}

\newcommand*\di{\,\mathrm{d}}
\newcommand\D{\,\mathrm{d}}

\newcommand\R{\mathbb{R}}
\newcommand\N{\mathbb{N}}
\newcommand\bl{\left(}
\newcommand\br{\right)}
\newcommand\Norm[2][]{\left\| #2 \right\|_{#1}}
\newcommand\CHI{S}
\newcommand\Qalpha{\tilde{\alpha}}
\newcommand\Qbeta{\tilde{\beta}}
\newcommand\dphi{s}
\newcommand\uzero{U_0}
\newcommand\udphi{U_\dphi}
\newcommand\AV[2][]{\left| #2 \right|_{#1}}

\newcommand{\REV}[1]{#1}

\newtheorem{thm}{Theorem}
\newtheorem{cor}[thm]{Corollary}
\newtheorem{lem}[thm]{Lemma}
\newtheorem{prp}[thm]{Proposition}
\newtheorem{ass}{Assumptions}
\theoremstyle{definition}

\newtheorem{rem}[thm]{Remark}

\numberwithin{equation}{section}
\numberwithin{thm}{section}

\title{Analysis of nonlinear poroviscoelastic flows with discontinuous porosities}
\author{Markus Bachmayr$^{1,*}$}
\email{bachmayr@igpm.rwth-aachen.de}
\author{Simon Boisser\'ee$^1$}
\email{boisseree@igpm.rwth-aachen.de}
\author{Lisa Maria Kreusser$^2$}
\email{lmk54@bath.ac.uk}
\address{$^1$ Institut f\"ur Geometrie und Praktische Mathematik, RWTH Aachen University, Templergraben 55, 52056 Aachen, Germany \\ $^2$ Department of Mathematical Sciences, University of Bath, Bath BA2 7AY, United Kingdom}
\date{\today}
\thanks{$^*$ Author to whom any correspondence should be addressed}
\thanks{M.B.\ acknowledges funding by Deutsche Forschungsgemeinschaft (DFG, German Research Foundation) -- project numbers 233630050, 442047500 -- TRR 146, SFB 1481.
S.B.\ has been funded in part by the M3ODEL consortium at Johannes Gutenberg University Mainz. L.M.K.\ acknowledges support from Magdalene College, Cambridge (Nevile Research Fellowship).}

\begin{document}

\begin{abstract}
    Existence and uniqueness of solutions is shown for a class of viscoelastic flows in porous media with particular attention to problems with nonsmooth porosities. The considered models are formulated in terms of the time-dependent nonlinear interaction between porosity and effective pressure, which in certain cases leads to porosity waves. In particular, conditions for well-posedness in the presence of initial porosities with jump discontinuities are identified.
\end{abstract}

\maketitle

\section{Introduction}
In porous media flow models in geophysical applications, the porosity of the medium is often treated as a static quantity. This assumption, however, need not always be justified due to the ability of rocks to deform by compaction.
In certain cases, the interaction of porosity and pressure can lead to the formation of \emph{porosity waves}. 
These can take the form of solitary waves formed by travelling higher-porosity regions \cite{Yarushina2015b} or of chimney-like channels \cite{Raess2019}. 
    
Such effects are important, for instance, in the modelling of rising magma \cite{McKenzie1984,Barcilon1986}, where porosity waves arise due to high temperatures. Such waves or channels can also form in soft sedimentary rocks, in salt formations or under the influence of chemical reactions; see for example \cite{Yarushina2015a,Raess2018,Raess2019}.
Quantifying uncertainties caused by the formation of preferential flow pathways can thus be important for safety analyses in geoengineering applications \cite{Yarushina2022}.

Here, we consider a poroviscoelastic model \REV{that generalizes} the one introduced in \cite{Connolly1998,Vasilyev1998} for the interaction of porosity and pressure. The equations for porosity $\phi$ and effective pressure $u$ studied numerically in \cite{Vasilyev1998} are of the basic form
\begin{equation}\label{eq:cp}
  \begin{aligned}
      \partial_t \phi & = - \nabla \cdot \phi^n (\nabla u + f) ,\\
      Q \partial_t u & = \nabla \cdot \phi^n (\nabla u + f) - \phi^m u ,
  \end{aligned}
\end{equation}
with a constant $Q>0$, fixed exponents $m,n\geq 1$ and a constant vector $f$, subject to initial conditions on $\phi$ and $u$ as well as boundary conditions on $u$. The problem \eqref{eq:cp} can be regarded as a slightly restricted version of the more general case considered in this work, but \eqref{eq:cp} has very similar features and is obtained by minor additional simplifications. 

Note that \eqref{eq:cp} can be regarded as the combination of a hyperbolic equation for $\phi$ and a parabolic equation for $u$. In the absence of strong smoothness assumptions on $u$, the advection coefficient in the first equation can be of low regularity. However, \eqref{eq:cp} can also be rewritten as
\begin{equation}\label{eq:cp2}
  \begin{aligned}
      \partial_t \phi & = - \phi^m u - Q \partial_t u ,\\
      Q \partial_t u & = \nabla \cdot \phi^n (\nabla u + f) - \phi^m u ,
  \end{aligned}
\end{equation}
where the parabolic problem for $u$ is coupled to a pointwise ordinary differential equation for $\phi + Q u$. For this reason, solutions of this problem turn out to be more well-behaved than \eqref{eq:cp} would suggest. The limiting case $Q=0$ in \eqref{eq:cp2}, corresponding to a purely viscous flow, takes the form 
\begin{equation}\label{eq:magma}
    \begin{aligned}
        \partial_t \phi & = - \phi^m u  ,\\
        0 & = \nabla \cdot \phi^n (\nabla u + f) - \phi^m u .
    \end{aligned}
  \end{equation}
This model has been studied in particular in the context of magma dynamics \cite{Wiggins1995,Simpson2006,Ambrose2018}.

For modelling sharp transitions between materials in \eqref{eq:cp2} and \eqref{eq:magma}, it is important to be able to treat porosities with \emph{jump discontinuities}. These turn out to be determined mainly by the initial condition of the form 
\[ \phi|_{t=0} = \phi_0 \,, \]
which arises in both problems, with a given function $\phi_0$ on the spatial domain. The initial data for $u$ that are additionally required in \eqref{eq:cp} and \eqref{eq:cp2} can be assumed to be sufficiently smooth in practically relevant models. For $\phi_0$ of low regularity, existence and uniqueness of solutions to  \eqref{eq:cp2} and \eqref{eq:magma} is not covered by existing results. While for \eqref{eq:magma}, existence and uniqueness of solutions have been obtained in \cite{Simpson2006} in one spatial dimension and in \cite{Ambrose2018} in the higher-dimensional case, these results require higher-order Sobolev regularity of $\phi_0$.

\subsection{Novel contributions} 
In the present work, for more general versions of both \eqref{eq:cp2} and the viscous limiting case \eqref{eq:magma}, we consider existence and uniqueness of solutions under low regularity requirements on the initial data $\phi_0$. In particular, we focus on conditions on $\phi_0$ that permit jump discontinuities. We arrive at different conditions for \eqref{eq:cp2} and \eqref{eq:magma}, depending on the dimension $d$ of the spatial domain $\Omega$.

For the purely viscous case as in \eqref{eq:magma}, when $d\in\{1,2\}$, we obtain existence and uniqueness of solutions up to the maximal time of existence under the sole requirement $\phi_0 \in L^\infty(\Omega)$. Using a different technique, we still obtain existence of solutions (where uniqueness remains open) for arbitrary $d$ provided that $\phi_0 \in BV(\Omega) \cap L^\infty(\Omega)$, which includes functions with discontinuities. Moreover, for arbitrary $d$ we also obtain a new result on existence and uniqueness assuming $\phi_0 \in C^{k,1}$ for some $k\geq 0$.   

For general poroviscoelastic problems, with \eqref{eq:cp2} as a special case, we obtain well-posedness for $\phi_0$ with jump discontinuities only under more restrictive assumptions:
with a partition of $\Omega$ into finitely many open subsets $\Omega^j$ with $C^{1,\mu}$-boundaries for a $\mu>0$, we assume that $\phi_0$ is H\"older continuous on $\overline{\Omega}^j$ for each $j$.
We then obtain existence and uniqueness of solutions where $\phi$ is also piecewise H\"older continuous in the spatial variables. These conditions cover typical model cases treated in applications.

In such model cases, initial porosities $\phi_0$ are often composed of regions corresponding to different materials separated by interfaces that, in relation to the considered length scales, need to be modelled as sharp transitions; it is thus crucial that the models remain well-posed in the limiting case of spatially discontinuous porosities. A second conclusion to be drawn from our results concerns the time evolution of discontinuities in $\phi$: since in all considered cases, we obtain $\phi \in C([0,T]; L^\infty(\Omega))$, the locations in $\Omega$ of jump discontinuities in $\phi$ cannot change with time -- or in other words, these locations are essentially determined by those in $\phi_0$.

We focus here on the case of bounded domains $\Omega$ with homogeneous Dirichlet conditions on $u$, which are frequently used in application scenarios (see for example \cite{Raess2018,Raess2019}), and on the well-posedness for potentially small time intervals. The existence of solutions on $\R^d$ for arbitrarily long times under the regularity assumptions considered here is left for future investigation.

\subsection{Outline}
In Section \ref{sec:modeldescription}, we describe the full viscoelastic model and its viscous limiting case; \REV{in particular, we state in Section \ref{sec:finalform} the general form of the equations that we then proceed to investigate.} Section \ref{sec:viscouslimit} is devoted to the local well-posedness (up to maximal times of existence) of solutions of the viscous model under different regularity assumptions, \REV{with the main results summarized in Theorems \ref{th:wellposed}, \ref{th:smoothwellposed} and \ref{th:bvwellposed}.} In Section \ref{sec:viscoelasticmodel}, we study the local well-posedness of the viscoelastic model, \REV{with the main result stated in Theorem \ref{th:wellposedPara}.}

\subsection{\REV{Notation}}

\REV{We denote the set of positive integers by $\N$ and the set of nonnegative integers by $\N_0$. Further, we denote the set of positive real numbers by $\mathbb R^+$. As usual, we denote by $C^k(\mathbb R^+)$ for any $k\in \N_0$ the space  of $k$-times continuously differentiable functions on $\mathbb R^+$. For the space of $k$-times continuously differentiable functions with compact support in $\Omega$, we write $C^k_0(\Omega)$. For any Banach space $X$ with norm $\|\cdot\|_{X}$, we denote by $C([0,T];X)$ the space of continuous functions $u\colon [0,T]\to X$  with $\|u\|_{C([0,T];X)}=\max_{s\in[0,T]} \|u(t)\|_X<\infty$. Analogously, the space of $k$-times continuously differentiable functions on $[0,T]$ with values in $X$ is denoted by $C^k([0,T];X)$. 
We denote by $C^{0,\gamma}(\overline \Omega)$ the space of Hölder continuous functions with Hölder exponent $\gamma$, and write $C^{0,\gamma}_\mathrm{loc}(\Omega)$ if the Hölder property is only satisfied locally in $\Omega$. The space $C^{k,\gamma}(\overline \Omega)$ consists of all $k$-times continuously differentiable functions whose $k$th partial derivatives are Hölder continuous with exponent $\gamma$. We denote by $BV(\Omega)$ the space of functions of bounded variation and by $L^p(\Omega)$ with $1\leq p\leq \infty$ the $L^p$-spaces on $\Omega$. We write $W^{k,p}(\Omega)$ with $1\leq p\leq \infty$ and $k\in \mathbb N_0$ for the Sobolev spaces on $\Omega$. We write $W^{k,p}_0$ for the closure of $C^\infty_0(\Omega)$ in $W^{k,p}(\Omega)$, where in particular, for $p=2$, $W^{-1,2}(\Omega)$ is dual space of $W^{1,2}_0(\Omega)$. For any real Banach space $X$ with norm $\|\cdot\|_{X}$ and any $1\leq p <\infty$, we denote by $L^p(0,T;X)$ the space of measurable functions $u\colon [0,T]\to X$ with $\|u\|_{L^p(0,T;X)}=(\int_0^T \|u(t)\|_X^p \di t)^{1/p}<\infty$.}

\section{Description of the model}\label{sec:modeldescription}

\REV{In this section, we introduce the general model equations that we focus on in this work. The models \eqref{eq:cp}, \eqref{eq:cp2} and \eqref{eq:magma} that have appeared in the literature can be obtained as special or limiting cases from these equations. We describe two such simplifications, the viscous limiting case and a small-porosity approximation. In addition, we introduce a transformed version of the general model that facilitates its analysis with nonsmooth data. We then state the precise mild-weak formulation of the problem that we investigate in the remainder of this work.}

\REV{Throughout, we assume a domain $\Omega\subseteq \R^d$ with $d \in \N$ to be given.} For ease of notation, for any $T>0$, we introduce the space-time cylinder $\Omega_T=(0,T) \times \Omega$.

\subsection{\REV{General poroviscoelastic model}}

The model for poroviscoelastic flow on which we focus in this work reads
\begin{subequations}\label{eq:model}
    \begin{align}
        \label{eq:modelphieq}	\partial_t \phi&=-(1-\phi) \bl \frac{b(\phi)}{\sigma(u)}u +Q\partial_t u \br ,\\
        \partial_t u&=\frac{1}{Q} \bl \nabla \cdot a(\phi)(\nabla u + (1-\phi) f)-\frac{b(\phi)}{\sigma(u)}u \br ,
    \end{align}
\end{subequations}
with functions $a$, $b$ and $\sigma$ that are to be specified, and where $Q>0$ and $f\in \R^d$ are assumed to be given constants.
Physically meaningful solutions of this problem should satisfy $\phi \in (0,1)$ on $\Omega_T$.
The problem is supplemented with initial data
\begin{align}\label{eq:ic}
    \phi(0,x)=\phi_0(x), \quad u(0,x)=u_0(x),\quad x\in\Omega,
\end{align}
for given functions $\phi_0\colon \Omega\to (0,1)$ and $u_0\colon \Omega\to \R$. 
In addition, if $\partial \Omega \neq \emptyset$, we equip~\eqref{eq:model} with homogeneous Dirichlet boundary conditions for $u$.

The derivation of \eqref{eq:model} is sketched in Appendix \ref{sec:deriv}. The coefficient functions $a$ and $b$ of main interest are of the form
\begin{equation}\label{eq:CK}
		a(\phi) = a_0\phi^n, \qquad b(\phi) = b_0\phi^m
\end{equation}
with real constants $a_0, b_0 > 0$ and $n, m \geq 1$. This assumption on $a$ is motivated by the Carman-Kozeny relationship \cite{Costa2006} between the porosity $\phi$ and the permeability of the medium. The function $\sigma$ accounts for \emph{decompaction weakening} \cite{Raess2018,Raess2019} and $\sigma/\phi^m$ can be regarded as the effective viscosity. 

\begin{ass}\label{ass:sigma}
	We assume that $\sigma\in C^1(\R)$ satisfies
	\begin{align*}
		\sup_{v \in \R} \sigma(v) < \infty,\quad
		\inf_{v\in \R} \sigma(v) > 0, \quad
		\sigma' \geq 0 \text{ on $\R$},
	\end{align*}
	as well as
	\begin{align*}
		\inf_{v\in \R} \left\{\frac{1}{\sigma(v)}-\frac{v\sigma'(v)}{\sigma^2(v)}\right\}>0,\quad
		c_L=\sup_{v\in \R} \left\{\frac{1}{\sigma(v)}-\frac{v\sigma'(v)}{\sigma^2(v)}\right\}<+\infty.
	\end{align*}
\end{ass}

A trivial example for $\sigma$ in Assumptions~\ref{ass:sigma} is given by $\sigma(v) = c_0$ for all $v\in\R$ with a constant $c_0>0$, proposed in \cite{Vasilyev1998}. Another example for $\sigma$, suggested in \cite{Raess2018,Raess2019} and satisfying Assumptions~\ref{ass:sigma}, is an expression of the form
\begin{align}\label{eq:sigma}
    \sigma(v)=c_0 \bl 1 - c_1 \bl 1 + \tanh \bl -\frac{v}{c_2} \br \br \br ,\quad v\in \R,
\end{align}
which provides a phenomenological model for decompaction weakening.
Here $c_0>0$ is a positive constant, $c_1\in [0,\frac12)$ and $c_2>0$, where $1 + \tanh$ can be regarded as a smooth approximation of a step function taking values in the interval $(0,2)$. 
In most the well-studied case $c_1=0$, as considered in \cite{Vasilyev1998}, one observes the formation of porosity waves, whereas $c_1>0$ with appropriate problem parameters and initial conditions can lead to the formation of channels.

With parameters from the stated ranges, $\sigma$ as in \eqref{eq:sigma} satisfies Assumptions~\ref{ass:sigma}: clearly, $\sigma$ is bounded from below by $c_0(1-2c_1)$ and from above by $c_0$, and is non-decreasing with $\sigma'(v)= \tfrac{c_0 c_1}{c_2} (1-\tanh^2(-v/c_2))$, $\lim_{v\to \pm \infty} v\sigma'(v)=0$ and $\sigma(v)>v\sigma'(v)$ for all $v\in \mathbb{R}$.

In what follows, it will be convenient to write
\begin{equation}\label{eq:kappa}
  \kappa(v) = \frac{v}{\sigma(v)}\,.
\end{equation}
Note that $\kappa$ is Lipschitz continuous with Lipschitz constant $c_L$.

Starting from \eqref{eq:model}, we next consider the viscous limiting case, small-porosity approximations and reformulations that can become necessary for initial data of low regularity.

\subsection{Purely viscous model}
	
The purely viscous case corresponds to the limiting case $Q = 0$ in~\eqref{eq:model}.
Since the dynamics of~\eqref{eq:model} in the cases of interest are typically dominated by the viscous behaviour, this case can give some first insights. For $Q=0$, the equations \eqref{eq:model} become
\begin{subequations}\label{eq:viscousmodel}
    \begin{align}
        \partial_t \phi&=-(1-\phi){b(\phi)}{\kappa(u)},\label{eq:viscousphi}\\
        0&=\nabla \cdot a(\phi) (\nabla u + (1-\phi)f)-{b(\phi)}{\kappa(u)},
    \end{align}
\end{subequations}
again with Dirichlet boundary conditions for $u$.
In this case, we have an initial condition on $\phi$,
\[
  \phi(0,x) = \phi_0(x), \quad x \in \Omega ,
\]
but no initial condition on $u$.

\subsection{Small-porosity approximations}
For initial data with $\phi_0(x)\in(0,1]$ for $x \in \Omega$ and bounded $u$, for a classical solution to~\eqref{eq:model} one has $\phi\leq 1$ due to the presence of the factor $1-\phi$ in~\eqref{eq:modelphieq}. Replacing the factor $1-\phi$ in~\eqref{eq:model} by 1, which is a commonly used approximation for small porosities, we obtain the modified model 
\begin{subequations}\label{eq:modelmod}
    \begin{align}
        \partial_t \phi&=- \bigl( {b(\phi)}{\kappa(u)} +Q\partial_t u \bigr) ,\label{eq:modelmodphi}\\
        \partial_t u&=\frac{1}{Q} \bigl( \nabla \cdot a(\phi)(\nabla u +f)-{b(\phi)}{\kappa(u)} \bigr) \label{eq:modelmodu}.
    \end{align}
\end{subequations}
We again equip~\eqref{eq:modelmod} with Dirichlet boundary conditions for $u$ and supplement it with initial data
\eqref{eq:ic}.
The same approximation can be made in the viscous model~\eqref{eq:viscousmodel}, resulting in
\begin{subequations}\label{eq:viscousmodified}
    \begin{align}
        \partial_t \phi&=-{b(\phi)}{\kappa(u)},\\
        0&=\nabla \cdot a(\phi) (\nabla u+f)-{b(\phi)}{\kappa(u)}\label{eq:viscouselliptic}.
    \end{align}
\end{subequations}
For small $\phi$, it is typically assumed that the qualitative behaviour of solutions to~\eqref{eq:modelmod} and~\eqref{eq:viscousmodified} are similar to the ones of the original models~\eqref{eq:model} and~\eqref{eq:viscousmodel}, respectively.

\subsection{\REV{Log-transformed reformulation of the general case}}\label{sec:reformulations}
Let us note first that for the full coupled problem~\eqref{eq:model} with nonsmooth initial porosity $\phi_0$, the interpretation of the first equation~\eqref{eq:modelphieq} is not obvious, since it contains a term of the form $(1-\phi)\partial_t u$. When $\phi$ has jump discontinuities in the spatial variables, $\partial_t u$ in general only exists in the distributional sense, that is, as an element of $L^2(0,T; W^{-1,2}(\Omega))$. In this case, the product of the distribution $\partial_t u$ and  $1-\phi$, which is  not weakly differentiable, may not be defined.

However, the original problem~\eqref{eq:model} including the factor $1-\phi$ can be reduced to a similar form as~\eqref{eq:modelmod} by the following observation: ~\eqref{eq:modelphieq} can formally be rewritten as 
\begin{align*}
    \partial_t \log (1 - \phi) = {b(\phi)}{\kappa(u)} +Q\partial_t u .
\end{align*}
Introducing the new variable $\lambda = - \log(1-\phi)$, so that $\phi = 1 - e^{-\lambda}$, the system~\eqref{eq:model} can be written in the form
\begin{subequations}\label{eq:logtransformed}
    \begin{align}
        \partial_t \lambda &= - \bl {b(1 - e^{-\lambda})}{\kappa(u)} +Q\partial_t u \br , \\
        \partial_t u &= \frac{1}{Q} \bl \nabla \cdot a(1-e^{-\lambda})(\nabla u + e^{-\lambda} f)-{b(1 - e^{-\lambda})}{\kappa(u)} \br ,
    \end{align}
\end{subequations}
which has the same structure as~\eqref{eq:modelmod}. Physically meaningful solutions with $0<\phi <1$ are obtained precisely when $\lambda > 0$. As we shall see, the reformulation~\eqref{eq:logtransformed} is also advantageous for obtaining a weak formulation, and we will thus consider~\eqref{eq:model} in this form.

\subsection{\REV{Mild and weak formulations}}\label{sec:finalform}

\REV{We next introduce the basic notions of solutions for the different formulations of the problem that we consider in the following sections. The viscoelastic models~\eqref{eq:modelmod} and~\eqref{eq:logtransformed} are both} of the general form
\begin{subequations}\label{eq:generalstructure}
    \begin{align}
        \partial_t \varphi &= -{\beta(\varphi)}{\kappa(u)} - Q \partial_t u ,\label{eq:generalstructure_a}\\
        \partial_t u &= \frac{1}{Q} \bl \nabla \cdot \alpha(\varphi) (\nabla u + \zeta(\varphi))-{\beta(\varphi)}{\kappa(u)} \br ,\label{eq:generalstructure_b}
    \end{align}
\end{subequations}
where $\alpha, \beta$ and $\zeta$ are given locally Lipschitz continuous functions, \REV{with initial conditions $\varphi(0, \cdot) = \varphi_0$ and $u(0,\cdot) = u_0$ in $\Omega$.} 
Note that since $\varphi$ in \eqref{eq:generalstructure} is in general bounded from above and below, on this range the functions $\alpha, \beta$ and $\zeta$ satisfy a uniform Lipschitz condition. 

\REV{To give a meaning to these equations for data of low regularity (in particular, when only $\varphi_0 \in L^\infty(\Omega)$ is assumed), 
we write \eqref{eq:generalstructure_a} in integral form and consider \eqref{eq:generalstructure_b} in weak formulation. This leads us to
\begin{subequations}\label{eq:fullmildweak}
\begin{align}
    \varphi(t,\cdot) &= \varphi_0 + Qu_0 - Qu(t,\cdot) - \int_0^t \beta(\varphi) \, \kappa(u) \D s , &  & t \in [0,T], \label{eq:phi}  \\
    \partial_t u &= \frac{1}{Q} \bl \nabla \cdot \alpha(\varphi) (\nabla u + \zeta(\varphi))-\beta(\varphi) \, \kappa(u) \br & & \text{in $W^{-1,2}(\Omega)$ for a.e.\ $t \in [0,T]$,} \label{eq:u}
\end{align}
\end{subequations}
subject to Dirichlet boundary conditions for $u$ and initial data $u(0,\cdot)=u_0$ in $\Omega$ for some given $u_0 \in L^2(\Omega)$. In addition to Assumptions \ref{ass:sigma} on $\sigma$ (and hence, in view of \eqref{eq:kappa} on $\kappa$), we make the following assumptions on $\alpha$, $\beta$ and $\zeta$.}

\begin{ass}\label{ass:alphabeta}
	We assume that $\alpha, \beta, \zeta \in C^{0,1}_\mathrm{loc}(\R^+)$
	and that $\alpha$ is strictly positive on $\R^+$; in other words, for each $\delta > 0$ there exists an $\epsilon > 0$ such that for all $x \in [\delta,\infty)$ we have $\alpha(x) \geq \epsilon > 0$.
	Furthermore we assume that $\beta(x) \geq 0$ for each $x \in \R^+$.
\end{ass}

Similarly \REV{to \eqref{eq:generalstructure}}, the viscous models \eqref{eq:viscousmodel} or~\eqref{eq:viscousmodified} can be written as 
\begin{subequations}\label{eq:generalstructureviscous}
    \begin{align}
        \partial_t \varphi &= -{\beta(\varphi)}{\kappa(u)} , \label{eq:generalstructureviscous1}\\
        0 &= \nabla \cdot \alpha(\varphi) (\nabla u + \zeta(\varphi))-{\beta(\varphi)}{\kappa(u)}.\label{eq:generalstructureviscous2}
    \end{align}
\end{subequations}
\REV{We again write \eqref{eq:generalstructureviscous1} in integral form and consider \eqref{eq:generalstructureviscous2} in weak formulation. This results in the viscous limit model in mild formulation for \eqref{eq:generalstructureviscous1} and weak formulation for \eqref{eq:generalstructureviscous2},
\begin{subequations}\label{eq:generalviscousmildweak}
\begin{align}
	\varphi(t,\cdot) &= \varphi_0 - \int_0^t \beta(\varphi)\kappa(u) \D s, & & \text{for $t \in [0,T]$,} \label{eq:generalviscousmild} \\
	0&=\nabla \cdot \alpha(\varphi) (\nabla u + \zeta(\varphi))-\beta(\varphi)\kappa(u) & & \text{in $W^{-1,2}(\Omega)$}, \label{eq:generalviscousweak}
\end{align}
\end{subequations}
where we assume that $\sigma$ satisfies Assumptions~\ref{ass:sigma} and $\alpha$, $\beta$ and $\zeta$ satisfy Assumptions~\ref{ass:alphabeta}.}

\begin{rem}
	Note that our existence and uniqueness results on the viscous model \eqref{eq:generalstructureviscous} and the fully viscoelastic model \eqref{eq:generalstructure} in Sections~\ref{sec:viscouslimit} and~\ref{sec:viscoelasticmodel}, respectively, also hold under weaker assumptions on $f$. In general, $f$ can be a vector-valued function in $L^\infty(\Omega)$, and for some of our results, even $f\in L^p(\Omega)$ with some $p<\infty$ is sufficient. Only for the results in Section~\ref{sec:higherreg} and the main results of Section~\ref{sec:parabolic}, we need higher regularity assumptions on $f$, still without requiring $f$ to be constant.
\end{rem}

\section{Analysis of the viscous limiting case}\label{sec:viscouslimit}

In this section, we investigate the existence and uniqueness of solutions to the viscous limit model \REV{\eqref{eq:generalviscousmildweak}}.
The proof is based on a reformulation as a fixed point problem for $\varphi$ similarly to the one considered in \cite{Ambrose2018,Simpson2006}, where the main difficulties in the present case result from our substantially lower regularity assumptions on the problem data.
As a first step, we consider the elliptic problem \eqref{eq:generalviscousweak} for fixed $\varphi$.

\subsection{Elliptic problem}

Let $T>0$ be given. As a first step, for any $t\in [0,T]$, we show the existence of a unique weak solution $u$ of the elliptic equation~\eqref{eq:generalviscousweak}, that is,
\begin{align}\label{eq:viscousellipticproof}
    \nabla \cdot \alpha(\tilde\varphi) (\nabla u + \zeta(\tilde\varphi))-\beta(\tilde\varphi) \kappa(u) = 0
\end{align}
with given $\tilde\varphi \in L^\infty(\Omega)$ such that $\inf_\Omega \tilde\varphi > 0$, for any space dimension $d\geq 1$.
	
\begin{thm}\label{th:existence_uniqueness_elliptic}
	Let $\Omega\subset \R^d$ be a bounded domain. Let $t\in [0,T]$ be given and suppose that $\tilde\varphi \in L^\infty(\Omega)$ 
	satisfies $\inf_\Omega \tilde\varphi >0$. 
	Then~\eqref{eq:viscousellipticproof} with homogeneous Dirichlet boundary conditions has a unique weak solution in $W^{1,2}_0( \Omega)$ satisfying 
	\begin{align}\label{eq:elliptic_W12bound}
	    \| u\|_{W^{1,2}(\Omega)}\leq C \Norm[L^{2}(\Omega)]{\alpha(\tilde\varphi) \, \zeta(\tilde\varphi)}
	\end{align}
	for a constant $C>0$ depending on $\Omega$, $\sigma$ and on the lower bound of $\tilde\varphi$, but independent of $u$.
\end{thm}

\begin{proof}
To prove the existence of a unique weak solution of~\eqref{eq:viscousellipticproof}, we consider the functional $\mathcal J\colon W^{1,2}_0(\Omega)\to \R$ defined by
\begin{align*}
	\mathcal J(u)=\frac{1}{2}\int_\Omega    \alpha(\tilde\varphi) \nabla u \cdot \nabla u\di x+\int_\Omega \beta(\tilde\varphi)V(u)\di x+\int_\Omega gu \di x.
\end{align*} 
Here, $g=-\nabla \cdot \alpha(\tilde\varphi) \, \zeta(\tilde\varphi) \in W^{-1,2}(\Omega)$ since $\alpha(\tilde\varphi) \, \zeta(\tilde\varphi) \in L^\infty(\Omega) \subseteq L^2(\Omega)$. We set 
\begin{align*}
	V(u)=\int_0^u \kappa(s)\di s
\end{align*}
for $\kappa$ from~\eqref{eq:kappa}.
\REV{By Assumptions \ref{ass:sigma},
\begin{align*}
	V''(u)=\kappa'(u)=\frac{1}{\sigma(u)}-\frac{u\sigma'(u)}{\sigma^2(u)}>0 ,
\end{align*}
and hence $V$ and $\mathcal{J}$ are strictly convex.} 

\REV{Since $\mathcal{J}$ is easily seen to be G\^ateaux differentiable by Lipschitz continuity of $\kappa$, a minimiser $u$ of $\mathcal J$ is a weak solution of~\eqref{eq:viscousellipticproof}, that is,
\begin{align*}
	\int_\Omega g\psi \di x=-\int_\Omega \nabla\psi \cdot \alpha(\tilde\varphi) \nabla u\di x-\int_\Omega \beta(\tilde\varphi) \kappa(u) \psi\di x
\end{align*}
for all $\psi\in W^{1,2}_0(\Omega)$. Since $\mathcal{J}$ is convex and G\^ateaux differentiable, it is also weakly lower semicontinuous (see \cite[Cor.~4.3]{Roubicek2013}).}

\REV{Moreover, $\mathcal J$ is coercive, that is, $\mathcal{J}(u)/\Norm[W^{1,2}_0]{u} \to \infty$ as $\Norm[W^{1,2}_0]{u}\to \infty$. To show this, we first note that $V(s)\geq 0$ for all $s\in \R$ and that $\tilde\varphi$ is bounded by a positive constant from below. This implies that there exist a constant $C_1>0$ such that $\mathcal J(u)\geq C_1\| \nabla u \|_{L^2(\Omega)}^2+\mathcal F(u)$,
where 
\begin{align*}
	\mathcal F(u)
	= \int_\Omega gu \di x=\int_\Omega\nabla u \cdot \alpha(\tilde\varphi) \, \zeta(\tilde\varphi)\di x\geq -\Norm[W^{1,2}(\Omega)]{u} \Norm[L^{2}(\Omega)]{\alpha(\tilde\varphi) \, \zeta(\tilde\varphi)} .
\end{align*}
By the Poincaré inequality, there exists a constant $C>0$ such that
\begin{align*}
	\mathcal J(u)\geq C \| u \|_{W^{1,2}(\Omega)}^2-\| u\|_{W^{1,2}(\Omega)} \|\alpha(\tilde\varphi) \, \zeta(\tilde\varphi)\|_{L^{2}(\Omega)} ,
\end{align*} 
and hence $\mathcal J$ is coercive.}

\REV{Since $\mathcal{J}$ is thus weakly lower semicontinuous, coercive, strictly convex and G\^ateaux differentiable, by \cite[Thm.~4.2]{Roubicek2013} there exists a unique $u \in W^{1,2}_0(\Omega)$ solving \eqref{eq:viscousellipticproof}.}
That the weak solution $u\in W^{1,2}_0(\Omega)$ of~\eqref{eq:viscousellipticproof} satisfies~\eqref{eq:elliptic_W12bound} follows from
\begin{align*}
	&\|u\|_{W^{1,2}(\Omega)} \|\alpha(\tilde\varphi) \, \zeta(\tilde\varphi)\|_{L^{2}(\Omega)} \\&\geq -\int_\Omega\nabla u \cdot \alpha(\tilde\varphi) \, \zeta(\tilde\varphi) \di x\\&=-\int_\Omega gu\di x=\int_\Omega \nabla u \cdot \alpha(\tilde\varphi) \nabla u\di x+\int_\Omega \frac{\beta(\tilde\varphi)}{\REV{\sigma( u)}} u^2\di x\geq C \| u\|_{W^{1,2}(\Omega)}^2,
\end{align*}
where $C>0$ depends on $\sigma$ and the lower bound of $\tilde\varphi$.
\end{proof}
 
Since we already know by Theorem~\ref{th:existence_uniqueness_elliptic} that there exists a solution to the nonlinear elliptic problem \eqref{eq:viscousellipticproof}, we consider $\sigma(u) \in L^\infty(\Omega)$ as being part of the coefficient and apply elliptic regularity results to \eqref{eq:viscousellipticproof}.

\subsubsection{Elliptic regularity theory}
We next recall two additional regularity results for elliptic problems that we apply both to \eqref{eq:viscousellipticproof} and to its linearisation with respect to $\varphi$.
We thus state these results for general linear elliptic equations 
\begin{align}\label{eq:ellgeneral}
	-\nabla \cdot (a \nabla u) + b u
	= 	\nabla \cdot f + g,
\end{align}
on a bounded domain $\Omega \subset \R^d$ with homogeneous Dirichlet boundary conditions,
where $a, b \in L^\infty(\Omega)$ and $\inf_\Omega a>0$.
Using \cite[Thm.~8.15]{Gilbarg2001}, we obtain uniform boundedness of the solution to \eqref{eq:ellgeneral} for $f$ and $g$ of suitable integrability.
\begin{thm}\label{th:unifbounded}
    Let $\Omega\subset \R^d$ be a bounded domain, $q>\max\{d,2\}$, $f \in L^q(\Omega,\R^d)$ and $g \in L^{q/2}(\Omega)$. Then the weak solution $u \in W^{1,2}_0(\Omega)$ to \eqref{eq:ellgeneral} satisfies
    \begin{align*}
        \Norm[L^\infty(\Omega)]{u}
        \leq C \bl \Norm[L^2(\Omega)]{u} + \REV{\Norm[L^q(\Omega,\R^d)]{f}} + \Norm[L^{q/2}(\Omega)]{g} \br,
    \end{align*}
    where $C>0$ depends on $d$, $q$, $\operatorname{meas}(\Omega)$, the upper bounds of $a, b$ and the lower bound of $a$.
\end{thm}
 
Under similar integrability assumptions on $f$ and $g$, we obtain higher integrability of the solution to \eqref{eq:ellgeneral} by the following result of Meyers in \cite[Thm.~1]{Meyers1963}. The assumptions on $\partial\Omega$ suffice as a consequence of \cite[Thm.~1]{Auscher2002}.
\begin{thm}\label{th:meyerhigherint}
     Let $\Omega\subset \R^d$ be a bounded domain with $C^1$-boundary.
     Then for the weak solution $u \in W^{1,2}_0(\Omega)$ of \eqref{eq:ellgeneral}, we have $u \in W^{1,p}(\Omega)$ for some $p>2$ if $f \in L^p(\Omega,\R^d)$ and $g \in L^{r}(\Omega)$ with $r = \max\{2, (dp)/(d+p)\}$, as well as the estimate
     \begin{align*}
     	\Norm[W^{1,p}(\Omega)]{u}
     	\leq C \bl \Norm[L^r(\Omega)]{u} + \Norm[L^p(\Omega)]{f} + \Norm[L^{r}(\Omega)]{g} \br.
     \end{align*}
\end{thm}
Note that we can always choose $p > 2$ sufficiently small such that $r = 2$.

\subsubsection{Lipschitz estimate with respect to porosity}

Applying Theorem~\ref{th:unifbounded} and Theorem~\ref{th:meyerhigherint} to~\eqref{eq:viscousellipticproof}, we obtain $u \in L^\infty(\Omega) \cap W^{1,p}(\Omega)$ for some $p > 2$ provided that $\zeta(\tilde\varphi) \in L^q(\Omega)$ for $q > \max\{2,d\}$.
Next, we show that the solution $u(\tilde\varphi)\in W^{1,2}_0(\Omega)$ to \eqref{eq:viscousellipticproof} is Lipschitz continuous with respect to $\tilde\varphi\in L^\infty(\Omega)$ for $d\leq 2$. We will use this Lipschitz estimate to study the coupled problem~\eqref{eq:generalviscousmildweak}.

\begin{cor}\label{th:lipschitzestimatelinf}
    Let $d \leq 2$ and suppose that $\Omega\subset \R^d$ is a bounded domain with $C^{1}$-boundary. Let \REV{$R>\epsilon >0$ and $\tilde\varphi_1, \tilde\varphi_2\in L^{\infty}( \Omega)$ with
	$\inf_\Omega\tilde\varphi_i \geq \epsilon > 0$ and $\Norm[L^\infty(\Omega)]{\tilde\varphi_i} \leq R$ for $i = 1,2$,}
    let $\alpha, \beta, \zeta \in C^1(\R^+)$ and let $\sigma \in C^{1,1}(\R)$. \REV{Let $u(\tilde\varphi_1)$ and $u(\tilde\varphi_2)$ be the unique weak solutions to~\eqref{eq:viscousellipticproof} in $W^{1,2}_0(\Omega)$ corresponding to $\tilde\varphi_1$ and $\tilde\varphi_2$, respectively.} Then
    \begin{align}\label{eq:lipschitzhoelder}
        \Norm[L^\infty(\Omega)]{u(\tilde\varphi_1)-u(\tilde\varphi_2)} \leq C \Norm[L^\infty(\Omega)]{\tilde\varphi_1-\tilde\varphi_2}
    \end{align}
	and
    \begin{align}\label{eq:lipschitzsigma}
    	\Norm[L^\infty(\Omega)]{\kappa(u(\tilde\varphi_1)) - \kappa(u(\tilde\varphi_2))} 
    	\leq C \Norm[L^\infty(\Omega)]{\tilde\varphi_1-\tilde\varphi_2}
    \end{align}
    with constants $C>0$ \REV{that may depend on $\epsilon$ and $R$, but not on $\tilde\varphi_1$ and $\tilde\varphi_2$.}
\end{cor}
\begin{proof}
As a first step, we show that the solution
$u \in L^\infty(\Omega)$ (by Theorem~\ref{th:unifbounded})
of \eqref{eq:viscousellipticproof} is Fr\'echet differentiable with respect to $\tilde\varphi \in L^\infty(\Omega)$ with $\inf_\Omega\tilde\varphi>0$.
For each $h \in L^\infty(\Omega)$, the candidate for the directional derivative $w_h = u'(\tilde\varphi)h \in W_0^{1,2}(\Omega)$ is given implicitly by the weak formulation of the equation
\begin{multline}\label{eq:elliptic_linear_wh}
	\nabla \cdot ( \alpha(\tilde\varphi) \nabla w_h) - \beta(\tilde\varphi) \, \kappa'(u(\tilde\varphi)) \, w_h\\
	= -\nabla \cdot \bl \alpha'(\tilde\varphi) h (\nabla u(\tilde\varphi)+\zeta(\tilde\varphi)) + \alpha(\tilde\varphi) \zeta'(\tilde\varphi) h \br + \beta'(\tilde\varphi) \, h \, \kappa(u(\tilde\varphi)).
\end{multline}
One obtains~\eqref{eq:elliptic_linear_wh} by formally differentiating 
\begin{align*}
    \frac{\di}{\di \dphi} \Bigl( L_{\tilde\varphi + \dphi h} (u(\tilde\varphi + \dphi h)) - g(\tilde\varphi + \dphi h) \Bigr)\bigg|_{\dphi=0} = 0, 
\end{align*}
where 
\begin{align*}
	L_{\tilde\varphi}(u(\tilde\varphi))
	&= \nabla \cdot \alpha(\tilde\varphi) \nabla u(\tilde\varphi)-\beta(\tilde\varphi) \kappa(u(\tilde\varphi)), \quad g(\tilde\varphi)=-\nabla \cdot \alpha(\tilde\varphi) \zeta(\tilde\varphi).
\end{align*}

We need to show that $w_h$ solving \eqref{eq:elliptic_linear_wh} is indeed a Fr\'echet derivative, that is, with $\udphi := u(\tilde\varphi + \dphi h)$ we have
\begin{align*}
	\lim_{\dphi \to 0} \Norm[L^\infty(\Omega)]{\frac{\udphi - \uzero}{\dphi} - w_h} = 0
\end{align*}
where the convergence is uniform in $h$ with $\Norm[L^\infty]{h}=1$, and $u'(\tilde\varphi)$ given by $w_h = u'(\tilde\varphi) h$ is a bounded operator on $L^\infty$.
Considering \eqref{eq:viscousellipticproof} for $\uzero$ and $\udphi$ yields the difference equation
\begin{multline}\label{eq:elldiff}
	\nabla \cdot (\alpha(\tilde\varphi) \nabla(\udphi - \uzero)) - \beta(\tilde\varphi) \Delta_{\kappa,\uzero}(\udphi) (\udphi - \uzero)\\
	\begin{split}
		=& - \nabla \cdot ((\alpha(\tilde\varphi + \dphi h) - \alpha(\tilde\varphi)) \nabla \udphi) + \kappa(\udphi) (\beta(\tilde\varphi + \dphi h) - \beta(\tilde\varphi))\\
		&- \nabla \cdot (\alpha(\tilde\varphi + \dphi h) \zeta(\tilde\varphi + \dphi h) - \alpha(\tilde\varphi) \zeta(\tilde\varphi)).
	\end{split}
\end{multline}
Here we define
\begin{align}\label{eq:Delta}
	\Delta_{\kappa,y}(x) :=
	\begin{cases}
		\frac{\kappa(x)-\kappa(y)}{x-y} &\text{if } x \neq y,\\
		\kappa'(y) &\text{else,}
	\end{cases}
\end{align}
which is used pointwise in~\eqref{eq:elldiff} and~\eqref{eq:ellderiv}.
We have  $\kappa(\udphi)-\kappa(\uzero)=\Delta_{\kappa,\uzero}(\udphi) (\udphi-\uzero)$ where $\Delta_{\kappa,\uzero}$ is continuous at $\uzero$ and $\Delta_{\kappa,\uzero}(\uzero)=\kappa'(\uzero)$.  

By combining \eqref{eq:elldiff} with \eqref{eq:elliptic_linear_wh}, we obtain 
\begin{multline}\label{eq:ellderiv}
	\nabla \cdot \bl \alpha(\tilde\varphi) \, \nabla\!\bl \frac{\udphi - \uzero}{\dphi} - w_h \br \br - \beta(\tilde\varphi) \kappa'(\uzero) \bl \frac{\udphi - \uzero}{\dphi} - w_h \br\\
	\begin{split}
		=& - \nabla \cdot \bl \bl \frac{\alpha(\tilde\varphi + \dphi h) - \alpha(\tilde\varphi)}{\dphi} - \alpha'(\tilde\varphi) h \br \nabla \udphi \br - \nabla \cdot (\alpha'(\tilde\varphi) h (\nabla \udphi - \nabla \uzero))\\
		&+ \bl \frac{\beta(\tilde\varphi + \dphi h) - \beta(\tilde\varphi)}{\dphi} - \beta'(\tilde\varphi) h \br \kappa(\udphi) + \beta'(\tilde\varphi) h (\kappa(\udphi) - \kappa(\uzero))\\
		&- \nabla \cdot \bl \frac{\alpha(\tilde\varphi + \dphi h) \zeta(\tilde\varphi + \dphi h) - \alpha(\tilde\varphi) \zeta(\tilde\varphi)}{\dphi} - \alpha'(\tilde\varphi) \zeta(\tilde\varphi) h - \alpha(\tilde\varphi) \zeta'(\tilde\varphi) h \br \\
		&+ \beta(\tilde\varphi) \frac{\udphi - \uzero}{\dphi} (\Delta_{\kappa,\uzero}(\udphi) - \kappa'(\uzero)).
	\end{split}
\end{multline}
In the case $d=1$, we apply Theorem~\ref{th:existence_uniqueness_elliptic} to~\eqref{eq:elldiff} to obtain
$\Norm[W^{1,2}(\Omega)]{\udphi - \uzero} \leq C \dphi$, which implies $U_s\to U_0$ in $L^\infty(\Omega)$ uniformly in $h$ with $\Norm[L^\infty]{h}=1$.
Applying Theorem~\ref{th:existence_uniqueness_elliptic} to~\eqref{eq:ellderiv} and using the properties of $\Delta_{\kappa,\uzero}$ at $\uzero$ yield the desired result.
If $d=2$, we apply Theorem~\ref{th:unifbounded} and Theorem~\ref{th:meyerhigherint} to~\eqref{eq:elldiff} implying that
$\Norm[L^\infty(\Omega)]{\udphi - \uzero} \leq C \dphi$ and $\udphi \to \uzero$ in $W^{1,p}(\Omega)$.
Applying Theorem~\ref{th:unifbounded} to~\eqref{eq:ellderiv} yields the desired result also in this case with convergence uniform in $h$ with $\Norm[L^\infty]{h}=1$.
Note that all the arguments here make use of the Lipschitz continuity of $\alpha, \beta, \zeta, \kappa$ and $\kappa'$.

To show that $u'(\tilde\varphi)$ is indeed a linear bounded operator on $L^\infty$, we again distinguish the cases $d=1$ and $d=2$. 
For $d=1$, the solution $w_h$ of the elliptic problem \eqref{eq:elliptic_linear_wh} satisfies  $\| w_h\|_{W^{1,2}(\Omega)}\leq C\Norm[L^\infty(\Omega)]{h}$ for some constant $C>0$ independent of $h$ and together with the embedding $\REV{W^{1,2}(\Omega)\subset L^\infty(\Omega) }$, an upper bound of $\Norm[L^\infty\to L^\infty]{u'(\tilde\varphi)}$ immediately follows.

To obtain such a bound for $d=2$, we apply Theorem~\ref{th:meyerhigherint} to obtain $\nabla u \in L^p(\Omega)$ for some $p>2 = d$. This allows us to apply Theorem~\ref{th:unifbounded}, which requires only $\nabla u \in L^p(\Omega)$ for some $p>d$, to the linearized problem \eqref{eq:elliptic_linear_wh}, and we obtain
\begin{align*}
	\Norm[L^\infty(\Omega)]{w_h}
	&\leq C_1 \Norm[L^\infty(\Omega)]{h} \bl \Norm[L^p(\Omega)]{\alpha'(\tilde\varphi) (\nabla u(\tilde\varphi) + \zeta(\tilde\varphi)) + \alpha(\tilde\varphi) \zeta'(\tilde\varphi)} + \Norm[L^p(\Omega)]{\beta'(\tilde\varphi) \kappa(u(\tilde\varphi)} \br\\
	&\leq C_2 \Norm[L^\infty(\Omega)]{h}
\end{align*}
with a constant $C_2>0$ independent of $h$, which shows $\Norm[L^\infty\to L^\infty]{u'(\tilde\varphi)} \leq C_2$.

We now show the Lipschitz estimates \eqref{eq:lipschitzhoelder} and \eqref{eq:lipschitzsigma}.
For $\tilde\varphi_1,\tilde\varphi_2\in L^\infty(\Omega)$, we set
\begin{align*}
    \theta(s) := u((1-s)\tilde\varphi_1 + s\tilde \varphi_2),
\end{align*}
and one can easily verify that its Fréchet derivative satisfies
\begin{align*}
    \theta' (s) &= u'((1-s)\tilde\varphi_1 +s\tilde\varphi_2) (\tilde\varphi_2 -\tilde \varphi_1).
\end{align*}
This allows us to use the fundamental theorem of calculus for the associated Bochner integral
\cite[Proposition A.2.3]{Liu2015} and
we obtain the Bochner integral representation 
\begin{align}\label{eq:integralrepell}
    u(\tilde\varphi_2) - u(\tilde\varphi_1) = \theta(1) - \theta(0) 
    = \int_0^1 \theta'(s)\,\di s = \int_0^1 u'((1-s)\tilde\varphi_1 +s\tilde\varphi_2) (\tilde\varphi_2 - \tilde\varphi_1)\,\di s.
\end{align}
This leads to~\eqref{eq:lipschitzhoelder} by taking $L^\infty$-norms on both sides. To show~\eqref{eq:lipschitzsigma}, we note the Lipschitz continuity of
$\kappa$ by Assumption~\ref{ass:sigma}, which immediately implies
\begin{align*}
    \Norm[L^\infty(\Omega)]{\kappa(u(\tilde\varphi_1)) - \kappa(u(\tilde\varphi_2))}
    \leq c_L \Norm[L^\infty(\Omega)]{u(\tilde\varphi_1)-u(\tilde\varphi_2)}
\end{align*}
and together with~\eqref{eq:lipschitzhoelder} yields~\eqref{eq:lipschitzsigma}.
\end{proof}

\subsection{Coupled problem with nonsmooth data}\label{sec:coupledviscous}

In this section, assuming $d\leq 2$, we investigate existence and uniqueness of the viscous limiting problem~\eqref{eq:generalviscousmildweak}. 
We consider the problem in a fixed point formulation in terms of $\varphi$ as in \cite{Simpson2006}, where this was used under substantially stronger regularity requirements.
 
Provided that $\varphi(t,\cdot)\in L^\infty(\Omega)$ is uniformly positive, the existence and uniqueness of a weak solution $u(\varphi(t,\cdot))\in W^{1,2}(\Omega)$ is guaranteed by Theorem~\ref{th:existence_uniqueness_elliptic}, and a Lipschitz estimate for solutions $u(\varphi(t,\cdot))\in W^{1,2}(\Omega)$ for such $\varphi(t,\cdot)\in L^\infty(\Omega)$ is satisfied by Corollary~\ref{th:lipschitzestimatelinf}.
For $T>0$ and $R>\epsilon>0$, let
\begin{equation}\label{eq:STdef}
\begin{aligned}
    \CHI_T=\Bigl\{ \varphi\in C([0,T]; L^\infty( \Omega))\colon &\sup_{t\in[0,T]} \|\varphi(t,\cdot) \|_{L^\infty( \Omega)}\leq R, \\  &\inf_{t\in[0,T]} \varphi(t,x) \geq \epsilon \text{ for a.e.~} x\in \Omega \Bigr\}.
\end{aligned}
\end{equation}
We define $\mathcal{E}$ as the nonlinear operator that maps $\varphi(t,\cdot)$ \REV{to $u(\varphi(t,\cdot))$, that is,}
\begin{equation}\label{eq:Edef}
 \REV{ \mathcal{E} \colon L^\infty(\Omega) \to W^{1,2}(\Omega), \; \tilde\varphi \mapsto u(\tilde\varphi)\,, }
\end{equation}
\REV{where $u(\tilde\varphi)$ is the solution of \eqref{eq:viscousellipticproof}.}
Moreover, for all $t\in [0,T]$, we define
\begin{equation}\label{eq:lambdaoperatorLinfty}
    \Lambda[\varphi](t,\cdot)
    = \varphi_0 - \int_{0}^{t} \beta(\varphi(s,\cdot)) \, \kappa(\mathcal{E}(\varphi(s,\cdot))) \D s.
\end{equation}
 Note that for $\varphi\in \CHI_T$, we only require continuity with respect to the time variable, but $\Lambda[\varphi]$ is continuously differentiable by construction. The operator $\Lambda$ satisfies a contraction property on $\CHI_T$.

 \begin{prp}\label{prp:localtimeexistence}
    Let $d \leq 2$ and $R>\epsilon>0$. Let $\Omega\subset \R^d$ be a bounded domain with $C^1$-boundary. \REV{In addition to Assumptions \ref{ass:sigma} and \ref{ass:alphabeta},} let $\alpha, \beta, \zeta \in C^1(\R^+)$. %
    Assume that $\varphi_0\in L^\infty(\Omega)$ satisfies
    $\|\varphi_0 \|_{L^\infty( \Omega)}< R$ and $ \inf_\Omega \varphi_0 > \epsilon$. Then the following statements hold.
    \begin{enumerate}
        \item[{\rm(i)}] There exists a $T_1 = T_1(R, \epsilon, d, \varphi_0)$ such that $\Lambda[\varphi] \in \CHI_T$ for all $T \in(0, T_1)$ and $\varphi \in \CHI_T$ where $u(t,\cdot)=\mathcal{E}(\varphi(t,\cdot))\in W^{1,2}(\Omega)$ for all $t\in [0,T]$ \REV{with $\mathcal{E}$ as in \eqref{eq:Edef}.}
        \item[{\rm(ii)}] There exists $T_2 \in (0,T_1]$, $T_2 = T_2(R, \epsilon, d, \varphi_0)$ such that $\Lambda$ defines a contraction on $\CHI_T$ for $T \in (0, T_2)$ with respect to the norm of $C([0,T]; L^\infty(\Omega))$.
    \end{enumerate}
\end{prp}

\begin{proof}
The proof follows a similar strategy as the one of \cite[Proposition~2.9]{Simpson2006}.
To prove (i), we need to find $T_1 > 0$ such that for each fixed $T\in(0, T_1)$ we have
\begin{align}
    \Lambda[\varphi]&\in C(\REV{[0,T]}; L^\infty(\Omega)),\label{equ8b}\\
    \|\Lambda[\varphi](t,\cdot) \|_{L^\infty(\Omega)}&\leq R,\label{equ6b}\\ \Lambda[\varphi](t,\cdot) &\geq \epsilon\label{equ7b}
\end{align}
for all $t \in [0,T]$. 

Let $T>0$ be given and let $\mathcal{E}(\varphi(t,\cdot)) \in W^{1,p}(\Omega)$ for $t \in [0,T]$ be the unique solution of
\begin{align*}
	L_{\varphi(t,\cdot)}(\mathcal{E}(\varphi(t,\cdot))) = - \nabla \cdot \alpha(\varphi(t,\cdot)) \, \zeta(\varphi(t,\cdot))
\end{align*}
from Theorem~\ref{th:existence_uniqueness_elliptic} together with Theorem~\ref{th:meyerhigherint}.
Furthermore we can estimate
\begin{align*}
    \Norm[L^\infty(\Omega)]{\Lambda[\varphi](t_1,\cdot) - \Lambda[\varphi](t_2,\cdot)}
    &\leq \left| \int_{t_1}^{t_2} \Norm[L^\infty(\Omega)]{\beta(\varphi(s,\cdot))\, \kappa(\mathcal{E}(\varphi(s,\cdot)))} \D s\right|\\
    &\leq C \AV{t_2 - t_1},
\end{align*}
for some constant $C>0$ independent of $T$, where we used Theorem~\ref{th:unifbounded} as well as the uniform bounds on $\varphi$ and $\sigma$. This implies that $\Lambda[\varphi] \in W^{1,\infty}(0,T;L^\infty(\Omega))$ and hence~\eqref{equ8b}.
For showing~\eqref{equ6b}, note that 
\begin{align*}
    \|\Lambda[\varphi](t,\cdot) \|_{L^\infty(\Omega)}&\leq \|\varphi_0 \|_{L^\infty(\Omega)}+ C T
\end{align*}
using the same arguments as above.
We may choose $T>0$ sufficiently small such that~\eqref{equ6b} holds because $\|\varphi_0 \|_{L^\infty(\Omega)}< R$ by assumption.

Next we show~\eqref{equ7b}. Since $|\Lambda[\varphi](t,\cdot)-\varphi_0|\leq CT$ for all $t\in [0,T]$ we have
\begin{align*}
    \Lambda[\varphi](t,\cdot)\geq \varphi_0- CT 
\end{align*}
by the assumption on $\varphi_0$, implying that we may choose $T>0$ sufficiently small so that also~\eqref{equ7b} holds. Hence, $T_1>0$ can be chosen such that~\eqref{equ8b}--\eqref{equ7b} are satisfied, which concludes the proof of (i).

To establish (ii), let $\varphi_1,\varphi_2\in\CHI_T$. For $t\in[0,T]$, we have
\begin{align*}
	\|\Lambda[\varphi_1]( t,\cdot) - \Lambda[\varphi_2](t,\cdot)\|_{L^\infty(\Omega)}
	&\leq C \int_{0}^{t} \|\varphi_1(s,\cdot) - \varphi_2(s,\cdot)\|_{L^\infty(\Omega)} \D s\\&\leq C T\sup_{s\in[0,T]} \|\varphi_1(s,\cdot) - \varphi_2(s,\cdot)\|_{L^\infty(\Omega)}
\end{align*}
for some constant $C>0$ independent of $T$ due to the Lipschitz continuity of
\REV{$\kappa$} and $\beta$, as well as the Lipschitz continuity of $u$ by \eqref{eq:lipschitzhoelder}. This implies that
\begin{align*}
	\sup_{t\in[0,T]}\|\Lambda[\varphi_1]( t,\cdot) - \Lambda[\varphi_2](t,\cdot)\|_{L^\infty(\Omega)}
	\leq C T\sup_{s\in[0,T]} \|\varphi_1(s,\cdot) - \varphi_2(s,\cdot)\|_{L^\infty(\Omega)}.
\end{align*}
By choosing $T_2 \in (0, T_1]$ such that $C T_2 < 1$ we obtain that $\Lambda$ is a contraction on $\CHI_T$ with respect to the $C([0,T];L^\infty\REV{(\Omega)})$-norm, which completes the proof of (ii).
\end{proof}

For bounded initial data $\varphi_0$ and $d \leq 2$, the contraction property in Proposition~\ref{prp:localtimeexistence} yields the local well-posedness of the solution to \eqref{eq:generalviscousmildweak}, which can be continued to its maximal time of existence.

\begin{thm}\label{th:wellposed}
	\REV{
	Let $d \leq 2$ and $R>\epsilon>0$, $\Omega\subset \R^d$ be a bounded domain with $C^1$-boundary. In addition to Assumptions \ref{ass:sigma} and \ref{ass:alphabeta}, let $\alpha, \beta, \zeta \in C^1(\R^+)$ and let  $\kappa$ be defined in terms of $\sigma$ as in \eqref{eq:kappa}. Assume that $\varphi_0\in L^\infty(\Omega)$ satisfies
	$\|\varphi_0 \|_{L^\infty( \Omega)}< R$ and $ \inf_\Omega \varphi_0 > \epsilon$.
	Let $S_T$ be defined as in \eqref{eq:STdef}, $\mathcal{E}$ as in \eqref{eq:Edef} and $\Lambda$ as in \eqref{eq:lambdaoperatorLinfty}.
    Then the following statements hold.
	\begin{enumerate}
	\item[{\rm(i)}] There exists $T>0$ depending on $R$, $\epsilon$, $d$ and $\varphi_0$ such that there exists a unique $\varphi \in \CHI_T$ solving $\varphi = \Lambda[\varphi]$. Moreover, for this $T$, $(\varphi, u)$ with $u(t,\cdot)=\mathcal{E}(\varphi(t,\cdot))\in W^{1,2}_0(\Omega)$ for $t\in[0,T]$ is the unique solution in $C([0,T]; L^\infty(\Omega))\times C([0,T]; W^{1,2}_0(\Omega))$ to
	\[
		\begin{aligned}
			\partial_t \varphi &= -{\beta(\varphi)}{\kappa(u)} ,\\
        0 &= \nabla \cdot \alpha(\varphi) (\nabla u + \zeta(\varphi))-{\beta(\varphi)}{\kappa(u)},
		\end{aligned}
	\]
	interpreted in the sense of~\eqref{eq:generalviscousmildweak}. Moreover, the solution $\varphi\in \CHI_T$ depends continuously in $C([0,T]; L^\infty(\Omega))$ on the initial data $\varphi_0$.
    \item[{\rm(ii)}] Either the solution can be extended to $(\varphi,u) \in C([0,\infty); L^\infty(\Omega))\times C([0,\infty); W^{1,2}_0(\Omega))$ or there exists a finite $T_{\mathrm{max}}>0$ such that
    \begin{align*}
        \lim_{T\to T_{\mathrm{max}}} \bl \|\varphi(T,\cdot) \|_{L^\infty( \Omega)}+\REV{\Norm[L^\infty(\Omega)]{\frac{1}{\varphi(T,\cdot)}}} \br =\infty. 
    \end{align*}
	\end{enumerate}}
\end{thm}

\begin{proof}
By Proposition~\ref{prp:localtimeexistence} there exists $T>0$ such that $\Lambda[\varphi]\in \CHI_T$ is a contraction on $\CHI_T$. The contraction mapping theorem implies that there exists a unique fixed point $\varphi \in \CHI_T$, which proves the existence and uniqueness of a local solution to~\eqref{eq:generalviscousmildweak} in $\CHI_T$ \REV{and thus the first statement of (i).}

\REV{To obtain the continuous dependence on the initial data, let $\varphi_0,\psi_0\in L^\infty(\Omega)$ satisfy the conditions of Proposition~\ref{prp:localtimeexistence}.}
Then, there exists $T>0$ such that $\varphi,\psi \in \CHI_T$ satisfy~\eqref{eq:generalviscousmildweak} with initial data $\varphi_0,\psi_0$. For $t\in[0,T]$, we have
\begin{align*}
    \|\varphi( t,\cdot) - \psi(t,\cdot)\|_{L^\infty(\Omega)}
    &\leq \|\varphi_0 - \psi_0\|_{L^\infty(\Omega)}+ C \int_{0}^{t} \|\varphi(s,\cdot) - \psi(s,\cdot)\|_{L^\infty(\Omega)} \D s
\end{align*}
for some $C>0$, independent of $\varphi,\psi$,
which yields
\begin{align*}
    \|\varphi( t,\cdot) - \psi(t,\cdot)\|_{L^\infty(\Omega)}
    &\leq \|\varphi_0 - \psi_0\|_{L^\infty(\Omega)}\exp(Ct)
\end{align*}
by Gr\"onwall’s inequality.

\REV{For (ii), suppose that there exists 
a maximal finite $T_{\text{max}}>0$
such that
\begin{align*}
    \sup_{t\in[0, T_{\text{max}})} \bl \|\varphi(t,\cdot) \|_{L^\infty(\Omega)}+\REV{\Norm[L^\infty(\Omega)]{\frac{1}{\varphi(T,\cdot)}}} \br \leq K 
\end{align*}
for some $K>0$. Then the local existence theory for $R=K+1$ and $\epsilon=\tfrac{1}{2}K^{-1}$ yields the existence of a solution on $[0,T_{\text{max}}+\delta]$ for some $\delta>0$, contradicting the maximality of $T_{\text{max}}$.}
\end{proof}

\subsection{Higher regularity}\label{sec:higherreg}
In this section, we show higher regularity for solutions of \eqref{eq:generalviscousmildweak} under additional smoothness assumptions on the problem data, in particular $\varphi_0 \in C^{k,1}(\overline{\Omega})$ for a $k \in \N_0$.
Again, we first consider auxiliary results for a general linear elliptic equation 
\begin{align}\label{eq:ellgeneral2}
	-\nabla \cdot (a \nabla u) + b u
	= 	f,
\end{align}
on a bounded domain $\Omega \subset \R^d$ with homogeneous Dirichlet boundary conditions,
where $a \in C^{0,1}(\overline{\Omega})$, $b \in L^\infty(\Omega)$ and $\inf_\Omega a>0$.
Using \cite[Thm.~9.15,~9.17]{Gilbarg2001}, we obtain the following $W^{2,p}$-bound for the weak solution of~\eqref{eq:ellgeneral2}.

\begin{thm}\label{th:existence_p}
    Let $d\geq 1$, let $\Omega \subset\R^d$  be a bounded domain with $C^{1,1}$-boundary. Let $t\in [0,T]$ be given. Suppose that  $f \in L^p(\Omega)$ with $p \in (1,\infty)$. Then~\eqref{eq:ellgeneral2} has a unique solution  $u\in W^{2,p}( \Omega) \cap W^{1,2}_0(\Omega)$ satisfying 
    \begin{align}\label{eq:elliptic_W2pbound}
	        \Norm[W^{2,p}(\Omega)]{u}
	        \leq C \Norm[L^p(\Omega)]{f}
	    \end{align}
    for a constant $C>0$ depending on $\Omega$, but independent of $u$.
\end{thm}

Under appropriate smoothness assumptions, the regularity of the solution to~\eqref{eq:ellgeneral2} can be further improved.

\begin{thm}\label{th:existence_wkp}
    Let $d\geq 1$ and suppose that $\Omega \subset\R^d$ be a bounded domain with $C^{k+1,1}$-boundary for a given $k\in\N_0$. Let $t\in [0,T]$ and let $a \in C^{k,1}(\overline{\Omega})$ be bounded from below by a positive constant.
    In addition, suppose that $b \in W^{k,\infty}(\Omega)$, $f \in W^{k,p}(\Omega)$ with $p \in (1,\infty)$. Then,~\eqref{eq:ellgeneral2} with homogeneous Dirichlet boundary conditions has a unique weak solution  $u\in W^{k+2,p}( \Omega)$ satisfying 
    \begin{align}\label{eq:elliptic_Wkpbound}
		\Norm[W^{k+2,p}(\Omega)]{u}
		\leq C \Norm[W^{k,p}(\Omega)]{f}
	\end{align}
    for a constant $C>0$ depending on $\Omega$ and $k$, but independent of $u$.
\end{thm}

The proof is a standard bootstrap argument following the lines of~\cite[Theorem~6.3.2 and~6.3.5]{Evans2010} and thus omitted.

\begin{rem}\label{rem:Wkp}
	Applying Theorem~\ref{th:existence_wkp} to~\eqref{eq:generalviscousweak}, we obtain $u \in W^{k+2,p}(\Omega)$ assuming that $\varphi(t,\cdot) \in C^{k,1}(\overline{\Omega})$ is bounded from below by a positive constant, $\alpha, \zeta \in C^{k,1}_\mathrm{loc}(\R^+)$, $\beta \in C^k(\R^+)$ and $\sigma \in C^{k,1}(\R)$. 
\end{rem}

By the Sobolev embedding $W^{k+2,p}(\overline \Omega) \subset C^{k+1,\gamma}(\overline{\Omega})$ with $\gamma = 1 - \frac{d}{p}$ for $p>d$ and  Theorem \ref{th:existence_wkp}, the solution $u\in W^{k+2,p}(\Omega)$ to the original elliptic problem \eqref{eq:generalviscousweak}  satisfies a Lipschitz estimate in $C^{k,1}(\overline\Omega)$ under slightly stronger assumptions on the coefficient functions.

\begin{prp}\label{prp:LipschitzEll}
    Let \REV{$d\geq 1$,} $p \in (d,\infty)$ and let $\Omega\subset \R^d$ be a bounded domain with $C^{k+1,1}$-boundary  for a given $k\in\N_0$. Let \REV{$R>\epsilon >0$ and} $\tilde\varphi_1, \tilde\varphi_2\in C^{k,1}(\overline{\Omega})$ \REV{such that $\inf_\Omega\tilde\varphi_i \geq \epsilon > 0$ and $\Norm[C^{k,1}(\overline\Omega)]{\tilde\varphi_i} \leq R$ for $i = 1,2$.}
    Furthermore, let $\alpha, \zeta \in C^{k+1,1}_\text{loc}(\R^+)$, $\beta \in C^{k+1}(\R^+)$ and let $\sigma$ be such that $\kappa \in C^{k+1,1}(\R)$. 
    Then the corresponding solutions $u(\tilde\varphi_1), u(\tilde\varphi_2)\in W^{k+2,p}(\Omega) \cap W^{1,2}_0(\Omega)$ to \eqref{eq:viscousellipticproof} satisfy
     \begin{align}\label{eq:lipschitzCk}
        \Norm[C^{k,1}(\overline{\Omega})]{u(\tilde\varphi_1)-u(\tilde\varphi_2)}
        \leq C_k \Norm[C^{k,1}(\overline{\Omega})]{\tilde\varphi_1-\tilde\varphi_2}
    \end{align}
    and
    \begin{align}\label{eq:Lipschitzhighordersigma}
    	\Norm[C^{k,1}(\overline{\Omega})]{\kappa(u(\tilde\varphi_1))-\kappa(u(\tilde\varphi_2))}
    	\leq C_k \Norm[C^{k,1}(\overline{\Omega})]{\tilde\varphi_1-\tilde\varphi_2}
    \end{align}
	with constants $C_k>0$ depending on $k$, \REV{$\epsilon$ and $R$, but not on $\tilde\varphi_1$ and $\tilde\varphi_2$.}

\end{prp}
\begin{proof}
    To show \eqref{eq:lipschitzCk}, we proceed similarly to Corollary~\ref{th:lipschitzestimatelinf}. For given uniformly positive $\tilde\varphi \in C^{k,1}(\overline{\Omega})$, we consider the linearized equation for $w_h$ with $h \in C^{k,1}(\overline{\Omega})$ which reads 
	\begin{multline*}
		\nabla \cdot ( \alpha(\tilde\varphi) \nabla w_h) - \beta(\tilde\varphi) \, \kappa'(u(\tilde\varphi)) \, w_h\\
		= -\nabla \cdot \bl \alpha'(\tilde\varphi) h (\nabla u(\tilde\varphi)+\zeta(\tilde\varphi)) + \alpha(\tilde\varphi) \zeta'(\tilde\varphi) h \br + \beta'(\tilde\varphi) \, h \, \kappa(u(\tilde\varphi))
	\end{multline*}
    as before, where  $\nabla u(\tilde\varphi) \in W^{k+1,p}(\Omega)$ due to Remark~\ref{rem:Wkp}. Hence, the requirements for Theorem~\ref{th:existence_wkp} are fulfilled. Furthermore, we can use ~\eqref{eq:elldiff} and~\eqref{eq:ellderiv} with right-hand sides in $W^{k,p}(\Omega)$ together with Theorem~\ref{th:existence_wkp} and the properties of $\Delta_{\kappa,\uzero}$ to show that
	\begin{align*}
		\lim_{\dphi \to 0} \Norm[C^{k,1}(\overline{\Omega})]{\frac{\udphi - \uzero}{\dphi} - w_h} = 0,
	\end{align*}
	\REV{where $\udphi = u(\tilde\varphi + \dphi h)$}. 
	The decisive step here, estimating the norm of $W^{k,p}(\Omega)$ by the one of $C^{k}(\overline{\Omega})$, is to show that $\Norm[C^k(\overline{\Omega})]{\Delta_{\kappa,\uzero}(\udphi) - \kappa'(\uzero)}$ converges to zero if $\dphi \to 0$, which is proved in Appendix~\ref{sec:Ckconv}.
	Using the Sobolev embedding $W^{k+2,p}(\overline{\Omega}) \subset C^{k+1,\gamma}(\overline{\Omega})$ with $\gamma = 1 - \frac{d}{p}$, this yields $\Norm[C^{k+1}]{w_h} \leq C_k \Norm[C^k]{h}$ for a constant $C_k>0$ depending on $k$, \REV{ $\epsilon$ and $R$, but not on $\tilde\varphi$.}

	This justifies the use of the Bochner integral representation of $u(\tilde\varphi_1) - u(\tilde\varphi_2)$ as in~\eqref{eq:integralrepell} and we can conclude that~\eqref{eq:lipschitzCk} holds.
    Since the $C^{k,1}$-norm contains derivatives for $k > 0$, we cannot apply the Lipschitz estimate for $\kappa$
    as in Corollary~\ref{th:lipschitzestimatelinf} to show that~\eqref{eq:Lipschitzhighordersigma} holds for some constant $C_k>0$ depending on $k$.
    To show \eqref{eq:Lipschitzhighordersigma}, we introduce 
    $\Theta(s) = \kappa(u((1-s)\tilde\varphi_1 + s \tilde\varphi_2))$
    with
    \begin{align*}
        \Theta'(s)
        = \kappa'(u((1-s)\tilde\varphi_1 + s \tilde\varphi_2)) \, u'((1-s)\tilde\varphi_1 +s\tilde\varphi_2) \, (\tilde\varphi_2 - \tilde\varphi_1).
    \end{align*}
    We then obtain $\Norm[C^{k,1}]{\Theta'(s)} \leq C_k \Norm[C^{k,1}]{\tilde\varphi_1 - \tilde\varphi_2}$ since $\kappa \in C^{k+1,1}(\R)$, where $C_k$ only depends on $k$, \REV{ $\epsilon$ and $R$, but not on $\tilde\varphi_1,\tilde\varphi_2$}. Using the associated Bochner integral representation yields \eqref{eq:Lipschitzhighordersigma}.
\end{proof}

Similarly to the space $\CHI_T$ in Section~\ref{sec:coupledviscous}, for $T>0$, $R>\epsilon>0$ and $k \in \N_0$ we define
\begin{align}\label{eq:STkdef}
    \CHI_T^k = \Bigl\{ \varphi\in C([0,T]; C^{k,1}(\overline \Omega))\colon \sup_{t\in[0,T]} \|\varphi(t,\cdot) \|_{C^{k,1}(\overline \Omega)}\leq R, \quad  \inf_{t\in[0,T]} \varphi(t,x) \geq \epsilon \text{ for all $x\in \Omega$}  \Bigr\}.
\end{align}

The  local well-posedness of the solution to \eqref{eq:generalviscousmildweak} in $\CHI_T^k$ can be shown  similar to Section~\ref{sec:coupledviscous} by establishing a contraction property with respect to the norm of $C([0,T]; C^{k,1}(\overline \Omega))$ as in Proposition~\ref{prp:localtimeexistence} and proceeding analogously to Theorem~\ref{th:wellposed}.
\begin{cor}\label{th:smoothwellposed}
	Let \REV{$d \geq 1$,}  $p \in (d,\infty)$ and suppose that  $R>\epsilon>0$. Let $\Omega\subset \R^d$ be a bounded domain with $C^{k+1,1}$-boundary for a given $k\in\N_0$.  Let  $\alpha, \beta, \zeta \in C^{k+1,1}_\mathrm{loc}(\R^+)$ satisfy Assumption~\ref{ass:alphabeta}. %
	Further, let $\sigma$ as in Assumption~\ref{ass:sigma} be such that \REV{for $\kappa$ as in \eqref{eq:kappa}, we have $\kappa \in C^{k+1,1}(\R)$,} and assume that $\varphi_0\in C^{k,1}(\overline{\Omega})$ satisfies
	$\Norm[C^{k,1}(\overline{\Omega})]{\varphi_0}< R$ and $\inf_\Omega \varphi_0 > \epsilon$. \REV{Let $S_T^k$ be defined as in \eqref{eq:STkdef}, $\mathcal{E}$ as in \eqref{eq:Edef} and $\Lambda$ as in \eqref{eq:lambdaoperatorLinfty}.
} Then \REV{the following statements hold.
	\begin{enumerate}
	\item[{\rm(i)}] There exists $T>0$ depending on $R$, $\epsilon$, $k$, $d$ and $\varphi_0$ such that 
	there exists a unique $\varphi \in \CHI_T^k$ solving $\varphi = \Lambda[\varphi]$. Moreover, for this $T$, $(\varphi, u)$ with $u(t,\cdot)=\mathcal{E}(\varphi(t,\cdot))\in W^{k+2,p}(\Omega)$ for $t\in[0,T]$ is the unique solution in $C([0,T]; C^{k,1}(\overline \Omega)) \times C([0,T]; W^{k+2,p}(\Omega))$ to
	\[
		\begin{aligned}
			\partial_t \varphi &= -{\beta(\varphi)}{\kappa(u)} ,\\
        0 &= \nabla \cdot \alpha(\varphi) (\nabla u + \zeta(\varphi))-{\beta(\varphi)}{\kappa(u)},
		\end{aligned}
	\]
	interpreted in the sense of~\eqref{eq:generalviscousmildweak}. 
    Moreover, the solution $\varphi\in \CHI_T^k$ depends continuously in $C([0,T]; C^{k,1}(\overline \Omega))$ on the initial data $\varphi_0$.
	\item[{\rm(ii)}] 
	Either the solution can be extended to $(\varphi,u) \in C([0,\infty); C^{k,1}(\overline \Omega)) \times C([0,\infty); W^{k+2,p}(\Omega))$  or there exists a finite $T_{\mathrm{max}}>0$ such that
   \begin{align*}
        \lim_{T\to T_{\mathrm{max}}} \bl \|\varphi(T,\cdot) \|_{C^{k,1}(\overline \Omega)}+\REV{\Norm[L^\infty(\Omega)]{\frac{1}{\varphi(T,\cdot)}}} \br =\infty.
    \end{align*}
    \end{enumerate}
    }
\end{cor}

\begin{proof}
Similarly as in Proposition~\ref{prp:localtimeexistence}, we  prove that there exists some $\tilde T>0$ such that $\Lambda$ in \eqref{eq:lambdaoperatorLinfty} defines a contraction on $\CHI_T^k$ for all $T\in (0,\Tilde{T})$ with respect to the norm of $C([0,T]; C^{k,1}(\overline \Omega))$. Proceeding as in the proof of Proposition~\ref{prp:localtimeexistence}, we obtain
\begin{align*}
    \Norm[C^{k,1}(\overline{\Omega})]{\beta(\varphi(s,\cdot)) \, \kappa(\mathcal{E}(\varphi(s,\cdot)))}
    \leq C
\end{align*}
uniformly by Theorem~\ref{th:existence_wkp} together with the boundedness of $\varphi$ and the submultiplicativity of $\Norm[C^{k,1}(\overline{\Omega})]{\cdot}$. The lower bound follows from the respective estimate of the $L^\infty$-norm in Proposition~\ref{prp:localtimeexistence}.
The contraction property for $\varphi_1,\varphi_2 \in S_T^k$ can be shown by considering
\begin{align*}
    &\Norm[C^{k,1}(\overline{\Omega})]{\beta(\varphi_1(s,\cdot)) \, \kappa(\mathcal{E}(\varphi_1(s,\cdot))) - \beta(\varphi_2(s,\cdot)) \, \kappa(\mathcal{E}(\varphi_2(s,\cdot)))}\\
    &\leq C \Norm[C^{k,1}(\overline{\Omega})]{\beta(\varphi_1(s,\cdot)) - \beta(\varphi_2(s,\cdot))}
    + C \Norm[C^{k,1}(\overline{\Omega})]{\kappa(\mathcal{E}(\varphi_1(s,\cdot))) - \kappa(\mathcal{E}(\varphi_2(s,\cdot)))}\\
    &\leq C \Norm[C^{k,1}(\overline{\Omega})]{\varphi_1(s,\cdot) - \varphi_2(s,\cdot)}
\end{align*}
which follows from Theorem~\ref{th:existence_wkp}, Proposition~\ref{prp:LipschitzEll} and the fact that $\beta \in C^{k+1,1}([\epsilon,R])$.

Using the contraction property of $\Lambda$ on $\CHI_T^k$, the \REV{statements (i) and (ii) now follow}  analogously to Theorem~\ref{th:wellposed}.
\end{proof}

\subsection{Existence for the coupled problem with nonsmooth initial data for general $d$}\label{sec:exEuler}

The results of Section \ref{sec:coupledviscous} yield existence and uniqueness of solutions under the minimal regularity requirement $\varphi_0 \in L^\infty(\Omega)$, but apply only to spatial dimensions $d \in \{1,2\}$.
To obtain existence of solutions with nonsmooth data for general $d$, we now use a different approach based on a time-discrete function-valued explicit Euler scheme for the viscous limit model~\eqref{eq:generalviscousmildweak}. Under similar assumptions as in Section \ref{sec:coupledviscous}, but under the slightly stronger regularity requirement $\varphi_0 \in L^\infty(\Omega) \cap BV(\Omega)$, we obtain existence of the solution to  \eqref{eq:generalviscousmildweak} for arbitrary $d$, as well as uniqueness for $d\leq 2$ and sufficiently small $T$.

We assume $\varphi_0 \in L^\infty(\Omega) \cap BV(\Omega)$ satisfying $\| \varphi_0\|_{L^\infty} < R$ and $\inf_\Omega \varphi_0 > \epsilon$ for some $R>\epsilon>0$ to be given and define the space
\begin{align*}
    \CHI
    = \Bigl\{ \varphi \in L^\infty(\Omega) \colon\, \Norm[L^\infty(\Omega)]{\varphi} \leq R, \varphi(x) \geq \epsilon \text{ for a.e.\ } x \in \Omega  \Bigl\}.
\end{align*}
We consider a time discretisation on $[0,T]$ with uniform grid points $t_k^N := k \tau$, where $\tau = \frac{T}{N}$.  We set 
$\varphi_0^N = \varphi_0 $
and define
\begin{align}\label{eq:exeul}
    \varphi_{k+1}^N
    = \varphi_k^N - \tau \beta(\varphi_k^N) \, \kappa(u(\varphi_k^N)),
\end{align}
where $u_k^N = u(\varphi_k^N)$ denotes the solution of~\eqref{eq:generalviscousweak} at time $t_k^N$. We obtain the following estimate of $\varphi_{k}^N$ in \eqref{eq:exeul}.

\begin{lem}
    Let $\Omega \subset \R^d$ be a domain and suppose that $\varphi_0 \in BV(\Omega)$ with $\| \varphi_0\|_{L^\infty} < R$ and $\inf_\Omega \varphi_0 > \epsilon$.
    Then there exists $T > 0$ such that $\varphi_k^N \in \CHI$ for all $N \in \N$ and $k = 1,\ldots,N$, and there exists $C > 0$ such that the following estimate holds,
    \begin{align}\label{eq:eulerestimate}
        \max_{k=0,\ldots,N}  \Norm[BV(\Omega)\cap L^\infty(\Omega)]{\varphi_{k}^N} \leq e^{C T} \Norm[BV(\Omega)\cap L^\infty(\Omega)]{\varphi_0}.
    \end{align}
\end{lem}

\begin{proof}
First of all we observe that by Theorem~\ref{th:unifbounded}, there exists $\bar{C}$ depending only on $\epsilon$ and $R$ such that whenever $\varphi_k^N \in \CHI$, we have
\begin{equation}\label{eq:Cbar}
	\Norm[L^\infty(\Omega)]{\beta(\varphi_k^N) \, \kappa(u(\varphi_k^N))}
	\leq C \Norm[L^\infty(\Omega)]{\varphi_k^N}
	\leq \bar{C} \,.
\end{equation}
By the Volpert chain rule for $BV$ functions (see for example~\cite{Volpert1967} as well as~\cite{Ambrosio1990}), for any $\psi \in BV(\Omega)$ and $F \in W^{1,\infty}(\R)$, we have $\Norm[BV(\Omega)]{F \circ \psi} \leq \Norm[L^\infty]{F'} \Norm[BV(\Omega)]{\psi}$. Applying this estimate, we make use of the Lipschitz continuity of $\kappa$ on $\R$ and of $\beta$ on $[\epsilon,R]$. A corollary of this chain rule is a product rule on $BV(\Omega)$ if both factors are also essentially bounded (see for example \cite[Sec.~B]{Boisseree2025}). Additionally using~\eqref{eq:elliptic_W12bound}, assuming that $\varphi_k^N \in \CHI$, we obtain
\begin{equation*}
\begin{aligned}
    \Norm[BV(\Omega)]{\beta(\varphi_k^N) \, \kappa(u(\varphi_k^N))}
    &\leq \Norm[BV(\Omega)]{\beta(\varphi_k^N)} \Norm[L^\infty(\Omega)]{\kappa(u(\varphi_k^N))} + \Norm[L^\infty(\Omega)]{\beta(\varphi_k^N)} \Norm[BV(\Omega)]{\kappa(u(\varphi_k^N))}\\
    &\leq C_1 \Big(\Norm[BV(\Omega)]{\varphi_k^N} \Norm[L^\infty(\Omega)]{u(\varphi_k^N)} + \Norm[L^\infty(\Omega)]{\varphi_k^N} \Norm[BV(\Omega)]{u(\varphi_k^N)} \Big)\\
    &\leq C_1 \Big(\Norm[BV(\Omega)]{\varphi_k^N} \Norm[L^\infty(\Omega)]{u(\varphi_k^N)} + \Norm[L^\infty(\Omega)]{\varphi_k^N} \Norm[W^{1,2}(\Omega)]{u(\varphi_k^N)} \Big)\\
	&\leq C_1 \Big(\Norm[BV(\Omega)]{\varphi_k^N} C_2 R + \Norm[L^\infty(\Omega)]{\varphi_k^N} C_3 R \Big)
    \leq C_4 \Norm[BV(\Omega)\cap L^\infty(\Omega)]{\varphi_k^N}\,
\end{aligned}
\end{equation*}
which together with \eqref{eq:Cbar} yields
\begin{equation}
	\begin{aligned}\label{bvstepboundviscous}
		\Norm[BV(\Omega)\cap L^\infty(\Omega)]{\beta(\varphi_k^N) \, \kappa(u(\varphi_k^N))}
		\leq C_5 \Norm[BV(\Omega)\cap L^\infty(\Omega)]{\varphi_k^N}\,.
	\end{aligned}
\end{equation}
We now choose the largest $T$ such that
\begin{equation*}
    \varphi_0 - T \bar C \geq \epsilon,\qquad
    \varphi_0 + T \bar C \leq R \,,
\end{equation*}
which ensures that $\varphi_k^N \in \CHI$ for all $N \in \N$ and $k = 1,\ldots,N$.
Then as a consequence of \eqref{bvstepboundviscous},
\begin{equation*}
    \Norm[BV(\Omega)\cap L^\infty(\Omega)]{\varphi_{k+1}^N} \leq \Norm[BV(\Omega)\cap L^\infty(\Omega)]{\varphi_k^N} + \tau C_5 \Norm[BV(\Omega)\cap L^\infty(\Omega)]{\varphi_k^N}
\end{equation*}
which implies \eqref{eq:eulerestimate}.
\end{proof}

From~\eqref{eq:exeul}, we  define two space-time functions
\begin{align}\label{eq:phisequencedef}
\begin{split}
    \varphi^N(t,\cdot)
    &= \frac{1}{\tau} \left[ (t_{k+1}^N - t) \varphi_k^N + (t - t_k^N) \varphi_{k+1}^N \right]\\
    \tilde{\varphi}^N(t,\cdot)
    &= \varphi_k^N
    \end{split}
\end{align}
for $t \in [t_k^N,t_{k+1}^N)$ for each $k$. The first observation is that
\begin{align*}
    \partial_t \varphi^N
    = \frac{\varphi_{k+1}^N - \varphi_k^N}{\tau}
     = - \beta(\varphi_k^N) \, \kappa(u(\varphi_k^N))
    = - \beta(\tilde{\varphi}^N) \, \kappa(u(\tilde{\varphi}^N))
\end{align*}
on each $(t_k^N,t_{k+1}^N)$, and thus on $[0,T]$. As a consequence,
\begin{align}\label{eq:timediscrete}
    \varphi^N(t,\cdot)
    = \varphi_0 - \int_0^t \beta(\tilde{\varphi}^N(s,\cdot)) \, \kappa(u(\tilde{\varphi}^N(s,\cdot))) \di s
\end{align}
a.e.\ in $\Omega$. The functions $\varphi^N, \tilde\varphi^N\colon [0,T]\to BV(\Omega)$ are Bochner measurable, since for each $N$, their image has finite dimension. Next, we prove the existence of solutions to \eqref{eq:generalviscousmildweak} by showing that there exists a strictly increasing $(N_n)_{n\in\N}$ such that the subsequences $\varphi^{N_n}, \tilde\varphi^{N_n}$ for $n\to \infty$ converge to the same limit $\varphi$ solving the integral formulation \eqref{eq:generalviscousmild}, that is,
 \begin{align*}
       \varphi(t,\cdot)
       = \varphi_0 - \int_0^t \beta(\varphi(s,\cdot)) \, \kappa(u(s,\cdot)) \di s \,.
\end{align*}

\begin{thm}\label{th:bvwellposed}
    Let \REV{$d\geq 1$ and $R>\epsilon>0$, and let} $\Omega \subset \R^d$ be a bounded domain. \REV{Assume that $\alpha, \beta, \zeta, \sigma$ satisfy Assumptions \ref{ass:sigma} and \ref{ass:alphabeta}, and let $\kappa$ be defined as in \eqref{eq:kappa}.} Suppose that $\varphi_0 \in BV(\Omega)$ with ${\| \varphi_0\|_{L^\infty} < R}$ and $\inf_\Omega \varphi_0 > \epsilon$, \REV{and let the sequences $\varphi^{N},\tilde\varphi^{N}$ for $N\in\N$ be defined as in \eqref{eq:phisequencedef} via 
    $\varphi_{k}^N$ in \eqref{eq:exeul} for $k=1,\ldots,N$.}
    Then, there exist   $T>0$, $\varphi \in W^{1,\infty}(0,T;BV(\Omega))$ \REV{and the corresponding $u\in L^\infty((0,T)\times \Omega)$ with $\sup_{t \in (0,T)} \| u(t,\cdot) \|_{W^{1,2}(\Omega)} < \infty$ such that	
    \[
		\begin{aligned}
			\partial_t \varphi &= -{\beta(\varphi)}{\kappa(u)} ,\\
        0 &= \nabla \cdot \alpha(\varphi) (\nabla u + \zeta(\varphi))-{\beta(\varphi)}{\kappa(u)},
		\end{aligned}
	\]
	interpreted in the sense of~\eqref{eq:generalviscousmildweak}.}       Moreover, there exists a strictly increasing $(N_n)_{n\in\N}$ such that the subsequences $\varphi^{N_n}, \tilde\varphi^{N_n}$ converge to $\varphi$ in $L^\infty(0,T;L^1(\Omega))$ \REV{and the subsequence $u(\tilde\varphi^{N_n})$ converges 
     weakly in $W^{1,2}(\Omega)$, strongly in $L^2(\Omega)$ and weak-* in $L^\infty(\Omega)$ to  $ u\in W^{1,2}(\Omega)\cap L^\infty(\Omega)$ a.e.\ in time}  as $n \to \infty$.
\end{thm}

\begin{proof}
From~\eqref{eq:eulerestimate} and \eqref{bvstepboundviscous},  we obtain
\begin{align*}
    \Norm[W^{1,\infty}(0,T;BV(\Omega)\cap L^\infty(\Omega))]{\varphi^N}
    &= \sup_{t \in [0,T]} \Norm[BV(\Omega)\cap L^\infty(\Omega)]{\varphi^N(t,\cdot)} + \sup_{t \in [0,T]} \Norm[BV(\Omega)\cap L^\infty(\Omega)]{\partial_t \varphi^N(t,\cdot)} \\
    & \leq ( 1 + C )  e^{C T} \Norm[BV(\Omega)\cap L^\infty(\Omega)]{\varphi_0}  .
\end{align*}
Note that $\varphi^N \in C([0,T];BV(\Omega))$ and $\partial_t \varphi^N \in L^\infty(0,T; BV(\Omega))$, where $BV(\Omega)$ is compactly embedded in $L^1(\Omega)$. By the Aubin-Lions theorem, there exists a subsequence, again denoted $\varphi^N$, that converges strongly in $C([0,T];L^1(\Omega))$ to some $\varphi \in C([0,T];BV(\Omega))$.
Moreover,
\begin{align*}
    \Norm[L^\infty(0,T;L^\infty(\Omega))]{\tilde{\varphi}^N - \varphi^N}
    &= \sup_{k=0,\ldots,N-1} \Norm[L^\infty(\Omega)]{\varphi_{k+1}^N - \varphi_k^N}\\
    &= \tau \sup_{k=0,\ldots,N-1} \Norm[L^\infty(\Omega)]{\beta(\varphi_k^N) \, \kappa(u(\varphi_k^N))}
    \leq \frac{T \bar{C}}{N}
\end{align*}
with $\bar C$ as in \eqref{eq:Cbar},
which implies that $\tilde{\varphi}^N \to \varphi$ in each space containing $L^\infty(0,T;L^1(\Omega))$.

Let us now consider the convergence of $u(\tilde{\varphi}^N)$ a.e.\ in time. We know that $\Norm[W^{1,2}(\Omega)]{u(\tilde{\varphi}^N)}$ is uniformly bounded. By Theorem \ref{th:unifbounded} we also obtain that $\Norm[L^\infty(\Omega)]{u(\tilde{\varphi}^N)}$ is uniformly bounded. Passing to a subsequence, $u(\tilde{\varphi}^N)$  converges weakly in $W^{1,2}(\Omega)$, strongly in $L^2(\Omega)$ and weak-* in $L^\infty(\Omega)$ to some $\tilde u \in W^{1,2}(\Omega)\cap L^\infty(\Omega)$. It remains to verify that $\tilde u = u(\varphi)$, that is,
\begin{align}\label{eq:weakformutilde}
   \int_\Omega \alpha(\varphi) \nabla \tilde u \cdot\nabla v + \beta(\varphi) \, \kappa(\tilde{u}) \, v \di x = \int_\Omega \alpha(\varphi) \, \zeta(\varphi) \cdot \nabla v\di x,\quad \text{for all $v \in W^{1,2}_0(\Omega)$.}
\end{align}
Since $\alpha$ is Lipschitz, $\alpha(\tilde{\varphi}^N)$  converges in $L^1(\Omega)$ and  has a subsequence (not renamed) that converges pointwise a.e.\ in $\Omega$. We consider
\begin{multline*}
	\int_\Omega \alpha(\varphi) \nabla \tilde u \cdot\nabla v \di x
	- \int_\Omega \alpha(\tilde{\varphi}^N) \nabla u(\tilde{\varphi}^N) \cdot\nabla v \di x  \\
	= \int_\Omega \alpha(\varphi) (\nabla \tilde u - \nabla u(\tilde{\varphi}^N))\cdot\nabla v\di x 
	+ \int_\Omega ( \alpha(\varphi) - \alpha(\tilde{\varphi}^N))  \nabla u(\tilde{\varphi}^N)\cdot\nabla v\di x.
\end{multline*}
Using  $\nabla u(\tilde{\varphi}^N) \rightharpoonup \nabla \tilde u$, we obtain that
\begin{align*}
    \int_\Omega \alpha(\varphi) (\nabla \tilde u - \nabla u(\tilde{\varphi}^N))\cdot\nabla v\di x \to 0.
\end{align*}
By the dominated convergence theorem, we have
\begin{align*}
  \left|  \int_\Omega ( \alpha(\varphi) - \alpha(\tilde{\varphi}^N))  \nabla u(\tilde{\varphi}^N)\cdot\nabla v\di x  \right|
   \leq \Norm[L^2]{(\alpha(\varphi)-\alpha(\tilde{\varphi}^N))\nabla v}\Norm[L^2]{\nabla u(\tilde{\varphi}^N)} \to 0,
\end{align*}
since $|\alpha(\varphi)-\alpha(\tilde{\varphi}^N)|^2|\nabla v|^2$ has an integrable majorant and goes to zero pointwise a.e.; the convergence of the other terms in \eqref{eq:weakformutilde} can be seen  similarly, but the arguments are  slightly easier because the involved sequences   converge strongly. Altogether, we  obtain $\tilde u = u(\varphi)$, and hence the convergence of $u(\tilde{\varphi}^N)$ to $u(\varphi)$ in the sense stated above.

Next we show that the right-hand side of~\eqref{eq:timediscrete} converges to the respective limit. The Lipschitz continuity of $\kappa$ yields strong convergence of $\kappa(u(\tilde{\varphi}^N))$ in $L^2(\Omega)$. Hence we have
\begin{align*}
    \Norm[L^1(\Omega)]{\beta(\varphi) \, \kappa(u(\varphi)) - \beta(\tilde{\varphi}^N) \, \kappa(u(\tilde{\varphi}^N))}
    \leq& \Norm[L^1(\Omega)]{\beta(\tilde{\varphi}^N) \bl \kappa(u(\varphi))  - \kappa(u(\tilde{\varphi}^N)) \br }\\
    &+ \Norm[L^1(\Omega)]{\kappa(u(\varphi)) \bl \beta(\varphi) - \beta(\tilde{\varphi}^N) \br }
\end{align*}
Both summands converge to zero due to the uniform $L^\infty$-bound on $\beta(\tilde{\varphi}^N)$ and the dominated convergence theorem together with the Lipschitz continuity of $\beta$  (since we can pass to a subsequence which converges a.e.\ in $\Omega$).
We use the dominated convergence theorem to show that the right-hand side of
\begin{multline*}
    \Norm[L^1(\Omega)]{\int_0^t \beta(\tilde{\varphi}^N(s,\cdot)) \, \kappa(u(\tilde{\varphi}^N(s,\cdot))) \di s - \int_0^t \beta(\varphi(s,\cdot)) \, \kappa(u(\varphi(s,\cdot))) \di s}\\
    \leq \int_0^t \Norm[L^1(\Omega)]{\beta(\tilde{\varphi}^N(s,\cdot)) \, \kappa(u(\tilde{\varphi}^N(s,\cdot))) - \beta(\varphi(s,\cdot)) \, \kappa(u(\varphi(s,\cdot)))} \di s
\end{multline*}
goes to zero for all $t \in [0,T]$ since its $L^\infty((0,T)\times \Omega)$-norm is bounded uniformly, which provides an integrable majorant. After  passing to an appropriate subsequence, this implies the convergence of the time integral in \eqref{eq:timediscrete} a.e.\ in $\Omega$.
\end{proof}\begin{rem}\REV{
	When $d \leq 2$, the solution from Theorem \ref{th:bvwellposed} coincides with the solution to \eqref{eq:generalviscousmildweak} of Theorem \ref{th:wellposed}, even though $BV(\Omega)$ is not a subspace of $L^\infty(\Omega)$ for $d\geq 2$:
for $d \leq 2$, let $\Omega \subset \R^d$ be a bounded domain with $C^1$-boundary. Suppose that $\varphi_0 \in BV(\Omega)$ with $\| \varphi_0\|_{L^\infty} < R$ and $\inf_\Omega \varphi_0 > \epsilon$.  The approximations $\varphi^N$ for $N\in\N$ given by \eqref{eq:exeul} converge in particular strongly in $C([0,T]; L^1(\Omega))$, which implies $\Norm[C({[0,T]};L^\infty(\Omega))]{\varphi} \leq R$. Then Theorems~\ref{th:wellposed} and \ref{th:bvwellposed} together imply that there exists a $T>0$ as well as a unique $\varphi \in W^{1,\infty}(0,T;BV(\Omega))$ and corresponding $u$ solving~\eqref{eq:generalviscousmildweak}.}
\end{rem}

\section{Analysis of the full viscoelastic model}\label{sec:viscoelasticmodel}
In this section, we investigate the existence and uniqueness of solutions to the full viscoelastic model \eqref{eq:fullmildweak}.
We \REV{continue to} assume that $\Omega \subset \R^d$ is a bounded domain \REV{with $d\in\N$} and we write $\Omega_T= (0,T]\times \Omega$ for a fixed time $T>0$.

To still obtain well-posedness of this problem for initial data $\varphi_0$ with jump discontinuities, in the present setting we need more restrictive regularity assumptions. 
As typical in applications of models \eqref{eq:fullmildweak}, we assume piecewise H\"older continuous initial data that may have jump discontinuities at boundaries of a partition of $\Omega$, which correspond to material interfaces.
Specifically, we shall assume that $\varphi_0$ is H\"older continuous on pairwise disjoint open subsets 
\begin{equation}\label{eq:partition}
	\Omega^j\subset \Omega,\quad j = 1,\ldots,M, 
\end{equation}
such that $\overline{\Omega} = \bigcup_{j=1}^M \overline{\Omega}^j$, $\Omega^j \Subset \Omega$ for $j = 1,\ldots,M-1$ and $\partial \Omega \subset \partial \Omega^M$. We set $\Omega_T^j = (0,T) \times \Omega^j$.

We first consider the parabolic problem \eqref{eq:u} separately for fixed $\varphi$.
We again assume that $\sigma$ satisfies Assumptions \ref{ass:sigma} and $\alpha, \beta$ and $\zeta$ satisfy Assumption~\ref{ass:alphabeta} in what follows.

\subsection{Parabolic problem}\label{sec:parabolic}

In this section, we study the existence and uniqueness of solutions of the parabolic problem  \eqref{eq:u}.
We set $\Qalpha=\tfrac{1}{Q}\alpha$ and $\Qbeta=\tfrac{1}{Q}\beta$. 
Weak solutions of~\eqref{eq:u} are characterized as 
$u\in L^2(0,T;W^{1,2}_0(\Omega))$ with weak derivatives $\partial_t u\in L^2(0,T; W^{-1,2}(\Omega))$ and $u(0)=u_0$ such that for every $\psi\in W^{1,2}_0(\Omega)$ and a.e.\ $t\in[0,T]$, 
\begin{multline}\label{eq:weakpara}
	\int_\Omega \partial_t u(t) \psi \di x +\int_\Omega \Qalpha(\varphi(t)) \, \nabla u(t) \cdot \nabla \psi\di x + \int_\Omega \Qbeta(\varphi(t)) \, \kappa(u(t))\psi \di x\\
	= \int_\Omega \psi \nabla \cdot \Qalpha(\varphi(t)) \, \zeta(\varphi(t)) \di x\,.
\end{multline}
Note that 
\begin{align*}
    \Qalpha(\varphi) &\in L^\infty((0,T)\times \Omega),\\
    \frac{\Qbeta(\varphi)}{\sigma(u)} &\in L^\infty((0,T)\times \Omega)\text{ for any } u \in L^2(0,T;W^{1,2}(\Omega)),\\
    \nabla \cdot \Qalpha(\varphi) \zeta(\varphi) &\in L^2(0,T; W^{-1,2}(\Omega)).
\end{align*}

The existence and uniqueness of solutions to the semilinear problem~\eqref{eq:weakpara} can be shown by standard techniques, see for example \cite[Chapter 7]{Evans2010}, adapted to our particular assumptions with source term $\nabla \cdot \Qalpha(\varphi) \, \zeta(\varphi) \in L^2(0,T;W^{-1,2}(\Omega))$. The proof is given in Appendix~\ref{sec:appexistence}.

\begin{thm}\label{th:exparabolic}
	Let $\Omega$ be a bounded domain and let $\varphi\in L^\infty((0,T)\times\Omega)$. 
	There exists a unique weak solution $u\in C([0,T]; L^2(\Omega))\cap L^2(0,T; W^{1,2}_0(\Omega))$ of \eqref{eq:weakpara}. Moreover, there exists a constant $C>0$, depending on $\Omega$, $T$, $Q$ and the coefficients of~\eqref{eq:weakpara}, such that
	\begin{multline}\label{eq:bound_un}
		\| u\|_{L^\infty(0,T;L^2(\Omega))} + \| u\|_{L^2(0,T;W^{1,2}_0(\Omega))} + \| \partial_t u \|_{L^2(0,T;W^{-1,2}(\Omega))}\\
		\leq C \bl \| \Qalpha(\varphi) \, \zeta(\varphi)\|_{L^2(0,T;L^{2}(\Omega))} +\|u_0\|_{L^2(\Omega)} \br.
	\end{multline}
\end{thm}

\subsubsection{Parabolic regularity theory}
We next consider auxiliary results, formulated for a general parabolic equation
\begin{align}\label{eq:pargeneral}
	\partial_t u - \nabla \cdot (a \nabla u) + b u
	= 	\nabla \cdot f + g
\end{align}
with homogeneous Dirichlet boundary conditions on $u$ with initial data $u_0$,
where $a, b \in L^\infty(\Omega)$ and $\inf_\Omega a>0$.
The boundedness of the solution to \eqref{eq:pargeneral} is ensured by the following proposition.
\begin{thm}\label{th:ubound}
	Let $\Omega$ be a Lipschitz domain
	and let $f, g \in L^\infty(\Omega_T)$ and $u_0 \in L^\infty(\Omega)$. Then
	\begin{align}\label{eq:supbound}
		\sup_\Omega |u(t, \cdot)| \leq \sup_\Omega |u_0| + C \bl \int_0^t \| u(s,\cdot)\|_{L^2(\Omega)}^2 \D s \br ^{1/2},
	\end{align}
	where $C>0$ is independent of $t$ and $T$. 
\end{thm}

\begin{proof}
	Specializing \cite[Theorem V.3.2]{DiBenedetto1993} to our setting, we get the desired result by choosing the parameters in~\cite{DiBenedetto1993} in the notation used there as follows:
	\begin{center}
		\begin{tabular}{lll}
			$p = 2$, & $\delta = 2$, & $\kappa_0 = 1$, \\
			$C_0 = \frac{1}{2} \mathrm{ess \, inf} \, a$, & $c_0 = 0$, & $\varphi_0 = \frac{1}{2 a} \, \AV{f}^2$, \\
			$C_1 = \mathrm{ess \, sup} \, a$, & $c_1 = 0$, & $\varphi_1 = \AV{f}$, \\
			$C_2 = 0$, & $c_2 = \mathrm{ess \, sup} \AV{b}$, & $\varphi_2 = \AV{g}$.
		\end{tabular}
	\end{center}
	Hence, $C$ depends on $d$ and the essential bounds of $a$, $b$, $f$ and $g$. 
\end{proof}

On the partition $\Omega^j$, $j=1,\ldots,M$, defined in \eqref{eq:partition}, we introduce parabolic Hölder spaces $C^{k,\gamma}_\mathrm{par}(\overline{\Omega}_T^j)$ for $k \in \{0,1\}$ and $\gamma \in (0,1)$ endowed with the parabolic space-time norm defined in~\cite{Ladyzenskaja1968},
\begin{equation}\label{eq:holder}
	\begin{aligned}
		\Norm[C^{0,\gamma}_\mathrm{par}(\overline{\Omega}^j_T)]{u}
		= &\sup_{\overline{\Omega}^j_T} |u|+ \sup_{\substack{t,x_1,x_2\\x_1 \neq x_2}} \frac{\AV{ u(t,x_1) - u(t,x_2)}}{\AV{x_1 - x_2}^\gamma} + \sup_{\substack{t_1,t_2,x\\t_1 \neq t_2}} \frac{\AV{ u(t_1,x) - u(t_2,x)}}{\AV{t_1 - t_2}^{\frac{\gamma}{2}}} ,\\
		\Norm[C^{1,\gamma}_\mathrm{par}(\overline{\Omega}^j_T)]{u}
		= &\sup_{\overline{\Omega}^j_T} |u| +
		\sup_{\overline{\Omega}^j_T} |\nabla u| + \sup_{\substack{t,x_1,x_2\\x_1 \neq x_2}} \frac{\AV{\nabla u(t,x_1) - \nabla u(t,x_2)}}{\AV{x_1 - x_2}^\gamma}\\
		&+ \sup_{\substack{t_1,t_2,x\\t_1 \neq t_2}} \frac{\AV{u(t_1,x) - u(t_2,x)}}{\AV{t_1 - t_2}^{\frac{1 + \gamma}{2}}} 
		+ \sup_{\substack{t_1,t_2,x\\t_1 \neq t_2}} \frac{\AV{\nabla u(t_1,x) - \nabla u(t_2,x)}}{\AV{t_1 - t_2}^{\frac{\gamma}{2}}} .
	\end{aligned}
\end{equation}
Assuming that $a$ in \eqref{eq:pargeneral} satisfies $a \in C^{0,\gamma}_\mathrm{par}(\overline{\Omega}_T^j)$ for $j=1,\ldots,M$, we obtain the following bound for the weak solution $u\in C^{1,\gamma}_\mathrm{par}(\overline{\Omega}_T^j)$ of \eqref{eq:pargeneral}.

\begin{thm}\label{th:parapiecewholder}
	Let $\Omega^j$ for $j=1,\ldots,M$, and hence also $\Omega$, have $C^{1,\mu}$-boundaries with a $\mu>0$, let $f \in C^{0,\gamma}_\mathrm{par}(\overline{\Omega}_T^j)$ for $j=1,\ldots,M$ with a $\gamma \in (0,\mu/(1+\mu)]$. If $u$ is a weak solution to \eqref{eq:pargeneral} with $u_0 \in C^{1,\gamma}(\overline{\Omega}^j)$ for $j=1,\ldots,M$ and homogeneous Dirichlet boundary conditions on $u$, then
	\begin{align}\label{eq:piecewiseest}
		\sum_{j=1}^{M} \Norm[C^{1,\gamma}_\mathrm{par}(\overline{\Omega}_T^j)]{u}
		&\leq C \Biggl( \Norm[L^\infty(\Omega_T)]{u} + \Norm[L^\infty(\Omega_T)]{g} + \sum_{j=1}^{M} \Norm[C^{0,\gamma}_\mathrm{par}(\overline{\Omega}_T^j)]{f}  + 
		\Norm[C^{1,\gamma}_\mathrm{par}(\overline{\Omega}_T^j)]{u_0}\Biggr) ,
	\end{align}
	where $C$ depends on $d$, $M$, $\gamma$, $\mu$, and on $\Omega^j$ and $\Norm[C^{0,\gamma}_\mathrm{par}(\overline\Omega_T^j)]{a}$ for $j=1,\ldots,M$. 
\end{thm}

\begin{proof}
	We first assume $u_0 \equiv 0$.
	Utilizing~\cite[Theorem~7.1]{Dong2021} and absorbing $b u$ into the right-hand side, we only have to apply the norm equivalences proved in Appendix~\ref{sec:equiv}, which hold for vanishing initial data. Note furthermore that $\Norm[L^\infty(\Omega_T)]{u}$ is bounded due to Theorem~\ref{th:ubound}.

	For nonvanishing $u_0$, we use its constant continuation in time and consider 
	\begin{align*}
		\partial_t v - \nabla \cdot (a \nabla (v+u_0)) + b \, (v+u_0)
		= \nabla \cdot f + g
	\end{align*}
	with vanishing initial data on $v = u - u_0$. The additional terms $\nabla \cdot a \nabla u_0$ and $b \, (v+u_0) \in L^\infty(\Omega_T)$ can be absorbed into the right-hand sides as before.
\end{proof}

\subsubsection{A parabolic Lipschitz estimate}

	By applying Theorem~\ref{th:ubound} and~\ref{th:parapiecewholder} to~\eqref{eq:weakpara} we get that $u \in C^{1,\gamma}_\mathrm{par}(\overline{\Omega}_T^j)$ if $\varphi \in C^{0,\gamma}_\mathrm{par}(\overline{\Omega}_T^j)$ for $j = 1,\ldots,M$.
Next, we show a Lipschitz bound for $u$.

\begin{lem}\label{lem:LipschitzParaPiecw}
	Assume that $\Omega^j$ has a $C^{1,\mu}$-boundary with $\mu>0$,
	Let \REV{$\epsilon >0$, $R \in \R$ and $\varphi_1, \varphi_2\in C^{0,\gamma}_\mathrm{par}(\overline{\Omega}_T^j)$ with $\inf_\Omega\varphi_i \geq \epsilon > 0$ and $\sum_{j=1}^{M} \Norm[C^{0,\gamma}_\mathrm{par}(\overline{\Omega}_T^j)]{\varphi_i} \leq R$ for $i = 1,2$,}
	$u(0,\cdot) = u_0 \in C^{1,\gamma}(\overline{\Omega}^j)$ for $j=1,\ldots,M$, $\gamma \in (0,\mu/(1+\mu)]$,
	$\alpha, \zeta \in C^2(\R^+)$,
	$\beta \in C^1(\R^+)$ and $\sigma \in C^{1,1}(\R)$. Let the unique weak solutions to~\eqref{eq:weakpara} for $\varphi_1$ and $\ \varphi_2$ be denoted by $u(\varphi_1)$ and $u(\varphi_2)$, respectively.
	Then
	\begin{align}\label{eq:pwlipschitzparabolic}
		\sum_{j=1}^{M} \Norm[C^{0,\gamma}_\mathrm{par}(\overline{\Omega}_T^j)]{u(\varphi_1) - u(\varphi_2)}
		\leq C T^{\gamma/2} \sum_{j=1}^{M} \Norm[C^{0,\gamma}_\mathrm{par}(\overline{\Omega}_T^j)]{\varphi_1 - \varphi_2}
	\end{align}
	and
	\begin{align}\label{eq:pwlipschitzparabolicsigma}
		\sum_{j=1}^{M} \Norm[C^{0,\gamma}_\mathrm{par}(\overline{\Omega}_T^j)]{\kappa(u(\varphi_1)) - \kappa(u(\varphi_2))}
		\leq C T^{\gamma/2} \sum_{j=1}^{M} \Norm[C^{0,\gamma}_\mathrm{par}(\overline{\Omega}_T^j)]{\varphi_1 - \varphi_2}
	\end{align}
	\REV{with constants $C>0$ that may depend on $\epsilon$ and $R$ but not on $\varphi_1,\varphi_2$ and $T$.}
\end{lem}

\begin{proof}
By considering the linearized equation
\begin{multline}\label{eq:parabolic_linear_wh}
	\partial_t w_h  - \nabla \cdot ( \Qalpha (\varphi) \nabla w_h)  + \Qbeta(\varphi) \, \kappa'(u(\varphi)) w_h\\
	= \nabla \cdot ( \Qalpha'(\varphi) h (\nabla u(\varphi) + \zeta(\varphi)) + \Qalpha(\varphi) \, \zeta'(\varphi) \, h ) - \Qbeta'(\varphi) h \, \kappa(u(\varphi))
\end{multline}
for the candidate for the directional derivative $w_h = u'(\varphi)h$
with $h \in C^{0,\gamma}_\mathrm{par}(\Omega_T^j)$, we can apply Theorem~\ref{th:parapiecewholder} (resp.~\cite[Theorem~7.1]{Dong2021}) using that a uniformly bounded solution exists by Theorem~\ref{th:ubound} and $\nabla u(\varphi) \in C^{0,\gamma}_\mathrm{par}(\Omega_T^j)$. 

To show that our candidate $w_h$ is actually a Fr\'echet derivative, that is, with $\udphi := u(\varphi + \dphi h)$ we have
\begin{align*}
	\lim_{\dphi \to 0} \sum_{j=1}^{M} \Norm[C^{0,\gamma}_\mathrm{par}(\overline{\Omega}_T^j)]{\frac{\udphi - \uzero}{\dphi} - w_h} = 0
\end{align*}
where the convergence is uniform in $h$ normalized in $C^{0,\gamma}_\mathrm{par}(\Omega_T^j)$. Subtracting the respective equations, we obtain
\begin{multline}\label{eq:pardiff}
	\partial_t (\udphi - \uzero) - \nabla \cdot (\Qalpha(\varphi) \nabla(\udphi - \uzero)) + \Qbeta(\varphi) \Delta_{\kappa, \uzero}(\udphi) (\udphi - \uzero)\\
	\begin{split}
		=& \nabla \cdot ((\Qalpha(\varphi + \dphi h) - \Qalpha(\varphi)) \nabla \udphi) - \kappa(\udphi) (\Qbeta(\varphi + \dphi h) - \Qbeta(\varphi))\\
		&+ \nabla \cdot (\Qalpha(\varphi + \dphi h) \zeta(\varphi + \dphi h) - \Qalpha(\varphi) \zeta(\varphi))
	\end{split}
\end{multline}
and
\begin{multline}\label{eq:parderiv}
	\partial_t \bl \frac{\udphi - \uzero}{\dphi} - w_h \br - \nabla \cdot \bl \Qalpha(\varphi) \, \nabla\!\bl \frac{\udphi - \uzero}{\dphi} - w_h \br \br + \Qbeta(\varphi) \kappa'(\uzero) \bl \frac{\udphi - \uzero}{\dphi} - w_h \br\\
	\begin{split}
		=& \nabla \cdot \bl \bl \frac{\Qalpha(\varphi + \dphi h) - \Qalpha(\varphi)}{\dphi} - \Qalpha'(\varphi) h \br \nabla \udphi \br + \nabla \cdot (\Qalpha'(\varphi) h (\nabla \udphi - \nabla \uzero))\\
		&- \bl \frac{\Qbeta(\varphi + \dphi h) - \Qbeta(\varphi)}{\dphi} - \Qbeta'(\varphi) h \br \kappa(\udphi) - \Qbeta'(\varphi) h (\kappa(\udphi) - \kappa(\uzero))\\
		&+ \nabla \cdot \bl \frac{\Qalpha(\varphi + \dphi h) \zeta(\varphi + \dphi h) - \Qalpha(\varphi) \zeta(\varphi)}{\dphi} - \Qalpha'(\varphi) \zeta(\varphi) h - \Qalpha(\varphi) \zeta'(\varphi) h \br \\
		&- \Qbeta(\varphi) \frac{\udphi - \uzero}{\dphi} (\Delta_{\kappa, \uzero}(\udphi) - \kappa'(\uzero)).
	\end{split}
\end{multline}
Here, $\Delta_{\kappa,\uzero}(\udphi)$ is defined as in~\eqref{eq:Delta} and is used pointwise in~\eqref{eq:pardiff} and~\eqref{eq:parderiv}.
As before we have  $\kappa(\udphi)-\kappa(\uzero)=\Delta_{\kappa,\uzero}(\udphi) (\udphi-\uzero)$ where $\Delta_{\kappa,\uzero}$ is continuous at $\uzero$ and $\Delta_{\kappa,\uzero}(\uzero)=\kappa'(\uzero)$.

First, we apply Theorem~\ref{th:ubound} and Theorem~\ref{th:parapiecewholder} to~\eqref{eq:pardiff} to show that $\Norm[L^\infty(\Omega_T)]{\udphi - \uzero} \leq C \dphi$ and $\udphi \to \uzero$ in $C^{1,\gamma}_{\mathrm{par}}(\overline{\Omega}^j_T)$-norm uniformly with respect to $h$ normalized in $C^{0,\gamma}_\mathrm{par}(\Omega_T^j)$.
Note that on the right-hand side of~\eqref{eq:piecewiseest}, $\Norm[L^\infty]{\Qbeta(\varphi + \dphi h) - \Qbeta(\varphi)}\to 0$ follows from the Lipschitz continuity of $\Qbeta$ and Theorem~\ref{th:ubound}.
In order to show
\begin{align*}
	\lim_{\dphi \to 0} \sum_{j=1}^M \Norm[C^{0,\gamma}_\mathrm{par}(\overline{\Omega}_T^j)]{\Qalpha(\varphi + \dphi h) - \Qalpha(\varphi)} = 0,
\end{align*}
we use~\cite[Thm. 3]{Goebel1992} twice for space and time. This yields (local) Lipschitz continuity of $\Qalpha$ with respect to $\Norm[C^{0,\gamma}_\mathrm{par}(\overline{\Omega}_T^j)]{\cdot}$ and hence proves the desired result since the convergence of the other divergence term follows analogously, this time using the regularity of $\zeta$ as well.

Next we apply Theorem~\ref{th:parapiecewholder} to~\eqref{eq:parderiv} where, as before, only the terms in divergence form pose problems. In order to prove
\begin{align*}
	\lim_{\dphi \to 0} \sum_{j=1}^M \Norm[C^{0,\gamma}_\mathrm{par}(\overline{\Omega}_T^j)]{\frac{\Qalpha(\varphi + \dphi h) - \Qalpha(\varphi)}{\dphi} - \Qalpha'(\varphi) h} = 0,
\end{align*}
we apply~\cite[Thm. 4]{Goebel1992}, again twice for space and time, to prove Fr\'echet differentiability of $\Qalpha$ with respect to $\Norm[C^{0,\gamma}_\mathrm{par}(\overline{\Omega}_T^j)]{\cdot}$. This proves the desired result since the convergence of the other divergence term follows analogously using the regularity of $\zeta$.

Since we obtain
\begin{align*}
	\sum_{j=1}^{M} \Norm[C^{1,\gamma}_\mathrm{par}(\overline{\Omega}_T^j)]{w_h}
	&\leq C \Bigg( \Norm[L^\infty(\Omega_T)]{\Qbeta(\varphi) \kappa'(u(\varphi)) w_h + \Qbeta'(\varphi) \, h \, \kappa(u(\varphi))}\\
	&\qquad+ \sum_{j=1}^{M} \Norm[C^{0,\gamma}_\mathrm{par}(\overline{\Omega}_T^j)]{\Qalpha'(\varphi) h (\nabla u(\varphi) + \zeta(\varphi)) + \Qalpha(\varphi) \, \zeta'(\varphi) \, h} \Bigg)\\
	&\leq C \sum_{j=1}^{M} \Norm[C^{0,\gamma}_\mathrm{par}(\overline{\Omega}_T^j)]{h},
\end{align*}
by applying Theorem~\ref{th:ubound} and~\ref{th:parapiecewholder} to~\eqref{eq:parabolic_linear_wh} this yields an upper bound of $\Norm{u'(\varphi)}$ for the associated piecewise operator norm. Since $w_h(0,\cdot)$ vanishes identically,
\begin{align*}
	\sum_{j=1}^{M} \Norm[C^{0,\gamma}_\mathrm{par}(\overline{\Omega}_T^j)]{w_h}
	\leq \max\{2, C\} \, T^{\gamma/2} \sum_{j=1}^{M} \Norm[C^{1,\gamma}_\mathrm{par}(\overline{\Omega}_T^j)]{w_h}
\end{align*}
by estimate~\eqref{eq:paraest} for a constant $C$ depending on the diameter of $\Omega$.

Considering the integral representation~\eqref{eq:integralrepell}, we obtain
\begin{align*}
	\sum_{j=1}^{M} \Norm[C^{0,\gamma}_\mathrm{par}(\overline{\Omega}_T^j)]{u(\varphi_1) - u(\varphi_2)}
	&= \sum_{j=1}^{M} \Norm[C^{0,\gamma}_\mathrm{par}(\overline{\Omega}_T^j)]{\int_0^1 u'((1-s) \varphi_2 + s \varphi_1) (\varphi_1 - \varphi_2) \D s}\\
	&\leq \int_0^1 \sum_{j=1}^{M} \Norm[C^{0,\gamma}_\mathrm{par}(\overline{\Omega}_T^j)]{u'((1-s) \varphi_2 + s \varphi_1) (\varphi_1 - \varphi_2)} \D s\\
	&\leq C T^{\gamma/2} \sum_{j=1}^{M} \Norm[C^{0,\gamma}_\mathrm{par}(\overline{\Omega}_T^j)]{\varphi_1 - \varphi_2}
\end{align*}
for some constant $C>0$ independent of $\varphi_1,\varphi_2$ and $T$. With
$\Theta(s) = \kappa(u((1-s)\varphi + s \psi))$,
we can thus again estimate $\Theta'$ as in Proposition~\ref{prp:LipschitzEll}.
Since $\kappa \in C^{1,1}(\R)$, this proves \eqref{eq:pwlipschitzparabolicsigma}.
\end{proof}

\subsection{Coupled problem}\label{sec:coupledfull}
In this section we investigate existence and uniqueness of the full viscoelastic model~\eqref{eq:fullmildweak} by a similar approach as in Sections~\ref{sec:coupledviscous} and~\ref{sec:higherreg}, that is, by establishing a contraction property on an appropriately chosen set.
Note that existence and uniqueness of the weak solution to the parabolic problem \eqref{eq:u} are provided in Theorem \ref{th:exparabolic} and
for a pointwise bound of the solution, Theorem~\ref{th:ubound} can be applied.

\REV{Given $T$ and $R>\epsilon>0$, let}
\begin{equation}\label{eq:STpwdef}
	\begin{aligned}
		\REV{\CHI_T^\mathrm{pw} =\Biggl\{ \varphi\in C(\REV{[0,T]};L^\infty(\Omega))\colon} &\REV{\sum_{j=1}^{M} \Norm[C^{0,\gamma}_\mathrm{par}(\overline{\Omega}_T^j)]{\varphi} \leq R,}\\ &\REV{\inf_{t\in[0,T]} \varphi(t,x) \geq \epsilon \text{ for a.e.~} x\in \Omega\Biggr\},}
	\end{aligned}
\end{equation}
endowed with the metric induced by the norm
\begin{align*}
	\Norm[\mathrm{pw}]{\varphi} = \sum_{j=1}^{M} \Norm[C^{0,\gamma}_\mathrm{par}(\overline{\Omega}_T^j)]{\varphi}.
\end{align*}
\REV{We define $\mathcal{P}$ as the nonlinear operator which maps $\varphi$ to $u(\varphi)$, that is,}
\begin{align}\label{eq:Pdef}
	\REV{\mathcal{P}\colon L^\infty((0,T)\times\Omega) \to C([0,T]; L^2(\Omega))\cap L^2(0,T; W^{1,2}_0(\Omega)),\quad \varphi \mapsto u(\varphi),}
\end{align}
\REV{where $u(\varphi)$ is the solution of~\eqref{eq:u}. For all $t \in [0,T]$, we define}
\begin{align}\label{eq:xioperator}
	\REV{\Xi[\varphi](t,\cdot)
	= \varphi_0 + Qu_0 - Q\mathcal{P}(\varphi)(t,\cdot) - \int_0^t \beta(\varphi(s,\cdot)) \, \kappa(\mathcal{P}(\varphi)(s,\cdot)) \D s \,.}
\end{align}
\REV{Here, $\varphi_0\in C^{0,\gamma}(\overline{\Omega}^j), u_0 \in C^{1,\gamma}(\overline{\Omega}^j)$ denote the initial conditions of $\varphi,u$, respectively.}
By the following proposition, \REV{$\Xi$ satisfies a contraction property on $\CHI_T^\mathrm{pw}$.}

\begin{prp}\label{prp:contr}
	Suppose that $\Omega^j$ has a $C^{1,\mu}$-boundary with $\mu>0$, $\varphi_0 \in C^{0,\gamma}(\overline{\Omega}^j)$, %
	$u(0,\cdot) = u_0 \in C^{1,\gamma}(\overline{\Omega}^j)$ for $j=1,\ldots,M$, $\gamma \in (0,\mu/(1+\mu)]$, $\alpha, \zeta \in C^2(\R^+)$, $\beta \in C^1(\R^+)$ and $\kappa \in C^{1,1}$.
	\REV{If in addition to Assumptions \ref{ass:sigma} and \ref{ass:alphabeta}}
	\begin{align*}
		\sum_{j=1}^{M} \Norm[C^{0,\gamma}(\overline{\Omega}^j)]{\varphi_0} < R
	\end{align*}
	and $\inf_\Omega \varphi_0 > \epsilon$,
	there exists a $\tilde T>0$\REV{, depending on $R$, $\epsilon$, $d$, $\varphi_0$, $u_0$,} such that $\Xi$ defines a contraction on $\CHI_T^\mathrm{pw}$ with respect to $\Norm[\mathrm{pw}]{\cdot}$ for all $T\in (0,\Tilde{T})$.
\end{prp}

\begin{proof}
We proceed as in the proof of Proposition~\ref{prp:localtimeexistence}. Note that $\sum_{j=1}^{M} \Norm[C^{0,\gamma}_\mathrm{par}(\overline{\Omega}_T^j)]{\varphi_0} < R$ for the constant continuation of $\varphi_0$ in time.
Furthermore, for the constant continuation of $u_0$ in time, note that
\begin{align*}
	\sum_{j=1}^{M} \Norm[C^{0,\gamma}_\mathrm{par}(\overline{\Omega}_T^j)]{u_0 - \mathcal{P}(\varphi)}
	\leq C T^{\gamma/2}
\end{align*}
by~\eqref{eq:paraest} since $\mathcal{P}(\varphi) \in C^{1,\gamma}_\mathrm{par}(\overline{\Omega}_T^j)$, and 
\begin{align*}
	\sum_{j=1}^{M} \Norm[C^{0,\gamma}_\mathrm{par}(\overline{\Omega}_T^j)]{\beta(\varphi) \, \kappa(\mathcal{P}(\varphi))}
	\leq C
\end{align*}
uniformly by Theorem~\ref{th:parapiecewholder}. Hence, using~\eqref{eq:intholder}, we can choose $T>0$ sufficiently small so that
\begin{align*}
	\sum_{j=1}^{M} \Norm[C^{0,\gamma}_\mathrm{par}(\overline{\Omega}_T^j)]{\Xi[\varphi]}\leq R
\end{align*}
holds. The lower bound of $\Xi[\varphi]$ follows
by repeating the previous estimates with the weaker $C(\REV{[0,T]};L^\infty(\Omega))$-norm which yields
\begin{align*}
	\Norm[C(\REV{[0,T]};L^\infty(\Omega))]{\Xi[\varphi] - \varphi_0}
	\leq C T^{\gamma/2}
\end{align*}
and hence
\begin{align*}
	\Xi[\varphi] \geq \varphi_0 - C T^{\gamma/2} 
\end{align*}
for all $t \in [0,T]$ a.e.\ in $\Omega$ similar to Proposition~\ref{prp:localtimeexistence}.
The contraction property can be shown by considering
\begin{align*}
	\sum_{j=1}^{M} \Norm[C^{0,\gamma}_\mathrm{par}(\overline{\Omega}_T^j)]{\mathcal{P}(\varphi_1) - \mathcal{P}(\varphi_2)}
	\leq C T^{\gamma/2} \sum_{j=1}^{M} \Norm[C^{0,\gamma}_\mathrm{par}(\overline{\Omega}_T^j)]{\varphi_1 - \varphi_2}
\end{align*}
as well as
\begin{align}\label{eq:parabolic_lipschitz_help}
\begin{split}
	\sum_{j=1}^{M} &\Norm[C^{0,\gamma}_\mathrm{par}(\overline{\Omega}_T^j)]{\int_0^t \beta(\varphi_1(s,\cdot)) \, \kappa(\mathcal{P}(\varphi_1)(s,\cdot)) - \beta(\varphi_2(s,\cdot)) \, \kappa(\mathcal{P}(\varphi_2)(s,\cdot)) \D s}\\
	&\leq C T^{1-\gamma/2} \sum_{j=1}^{M} \Norm[C^{0,\gamma}_\mathrm{par}(\overline{\Omega}_T^j)]{\beta(\varphi_1) \, \kappa(\mathcal{P}(\varphi_1)) - \beta(\varphi_2) \, \kappa(\mathcal{P}(\varphi_2))}\\
	&\leq T^{1-\gamma/2} \sum_{j=1}^{M} C_1 \Norm[C^{0,\gamma}_\mathrm{par}(\overline{\Omega}_T^j)]{\beta(\varphi_1) - \beta(\varphi_2)}
	+ C_2 \Norm[C^{0,\gamma}_\mathrm{par}(\overline{\Omega}_T^j)]{\kappa(\mathcal{P}(\varphi_1)) - \kappa(\mathcal{P}(\varphi_2))}\\
	&\leq T^{1-\gamma/2} (C_1 + C_2 T^{\gamma/2}) \sum_{j=1}^{M} \Norm[C^{0,\gamma}_\mathrm{par}(\overline{\Omega}_T^j)]{\varphi_1 - \varphi_2}
\end{split}	
\end{align}
which follows from~\eqref{eq:intholder}, Theorem~\ref{th:parapiecewholder}, Lemma~\ref{lem:LipschitzParaPiecw} and the fact that $\beta \in C^{0,1}([\epsilon,R])$, $\varphi_i \in C^{0,\gamma}_\mathrm{par}(\overline{\Omega}_T^j)$ for $i=1,2$ as well as uniform positivity of $\sigma$.
Hence we obtain that $\Xi$ defines a contraction on $\CHI_T^\mathrm{pw}$ if we choose $\tilde{T}$ small enough.
\end{proof}

Using Proposition~\ref{prp:contr}, we can show the existence of a unique solution to~\eqref{eq:fullmildweak} in $\CHI_T^\mathrm{pw}$ for sufficiently small $T>0$.
\begin{thm}\label{th:wellposedPara}
	\REV{
Let  $d\geq 1$, $R>\epsilon>0$ and let   $\Omega \subset \R^d$   be a bounded Lipschitz domain.
Let $\Omega^j\subset \Omega$ for $j = 1,\ldots,M$ be pairwise disjoint open subsets 
such that $\overline{\Omega} = \bigcup_{j=1}^M \overline{\Omega}^j$, $\Omega^j \Subset \Omega$ for $j = 1,\ldots,M-1$ and $\partial \Omega \subset \partial \Omega^M$.
	Suppose that $\Omega^j$ has a $C^{1,\mu}$-boundary with $\mu>0$, $\varphi_0 \in C^{0,\gamma}(\overline{\Omega}^j)$ and  $u(0,\cdot) = u_0 \in C^{1,\gamma}(\overline{\Omega}^j)$ for $j=1,\ldots,M$, $\gamma \in (0,\mu/(1+\mu)]$. In addition to Assumptions \ref{ass:sigma} and \ref{ass:alphabeta}, let $\alpha, \zeta \in C^2(\R^+)$, $\beta \in C^1(\R^+)$ and let $\kappa \in C^{1,1}$ be defined as in \eqref{eq:kappa}.
	Suppose that
	\begin{align*}
		\sum_{j=1}^{M} \Norm[C^{0,\gamma}(\overline{\Omega}^j)]{\varphi_0} < R,
	\end{align*}
	$\inf_\Omega \varphi_0 > \epsilon$,
	$S_T^\mathrm{pw}$ be defined as in \eqref{eq:STpwdef}, $\mathcal{P}$ as in \eqref{eq:Pdef} and $\Xi$ as in \eqref{eq:xioperator}.
	Then the following statements hold.
	\begin{enumerate}
		\item[{\rm(i)}] There exists $T>0$ depending on $R$, $\epsilon$, $d$, $\varphi_0$ and $u_0$ such that there exists a unique $\varphi \in \CHI_T^\mathrm{pw}$ solving $\varphi = \Xi[\varphi]$. Moreover, for this $T$, $(\varphi, u) \in C^{0,\gamma}_\mathrm{par}(\overline{\Omega}_T^j)\times C^{1,\gamma}_\mathrm{par}(\overline{\Omega}_T^j)$ for $j = 1,\ldots,M$  is the unique solution to
		\[
		\begin{aligned}
			\partial_t \varphi &= -{\beta(\varphi)}{\kappa(u)} - Q \partial_t u ,\\
			\partial_t u &= \frac{1}{Q} \bl \nabla \cdot \alpha(\varphi) (\nabla u + \zeta(\varphi))-{\beta(\varphi)}{\kappa(u)} \br ,
		\end{aligned}
		\]
		interpreted in the sense of~\eqref{eq:fullmildweak}. Moreover, the solution $\varphi\in \CHI_T^\mathrm{pw}$ depends continuously in $C^{0,\gamma}_\mathrm{par}(\overline{\Omega}_T^j)$ on the initial data $\varphi_0$.
		\item[{\rm(ii)}] Either the solution can be extended to $(\varphi,u) \in C^{0,\gamma}_\mathrm{par}([0,\infty) \times \overline{\Omega}^j ) \times C^{1,\gamma}_\mathrm{par}( [0,\infty) \times \overline{\Omega}^j )$ or there exists a finite $T_{\mathrm{max}}>0$ such that
		\begin{align*}
			\lim_{T\to T_{\mathrm{max}}} \bl \sum_{j=1}^{M} \Norm[C^{0,\gamma}(\overline{\Omega}^j)]{\varphi(T,\cdot)} + \REV{\Norm[L^\infty(\Omega)]{\frac{1}{\varphi(T,\cdot)}}} \br =\infty. 
		\end{align*}
		\end{enumerate}}
	%
\end{thm}

\begin{proof}
Note that we  fixed some $T>0$ in order to get a time-uniform  constant for the bound of the solution $u$ in Theorem~\ref{th:exparabolic}.
By Proposition~\ref{prp:contr} there exists $T_2 \in (0,T]$ such that $\Xi[\varphi] \in \CHI_{T_2}^\mathrm{pw}$ is a contraction on $\CHI_{T_2}^\mathrm{pw}$. The upper bound $T$ is needed here since the operator $\mathcal{P}$ may otherwise not be defined. The contraction mapping theorem implies that there exists a unique fixed point $\varphi \in \CHI_{T_2}^\mathrm{pw}$ which proves the existence and uniqueness of a local solution to~\eqref{eq:fullmildweak} in $\CHI_{T_2}^\mathrm{pw}$\REV{, and thus the first statement of (i).}

To show the continuous dependence on the initial data, we consider $\varphi_0,\psi_0 \in C^{0,\gamma}(\overline{\Omega}^j)$ satisfying
$\sum_{j=1}^{M} \Norm[C^{0,\gamma}(\overline{\Omega}^j)]{\varphi_0} < R$, $\inf_{\Omega_T}\varphi_0 > \epsilon$ and $\sum_{j=1}^{M} \Norm[C^{0,\gamma}(\overline{\Omega}^j)]{\psi_0} < R$, $ \inf_{\Omega_T} \psi_0 > \epsilon$.
Then, there exists a $T>0$ such that $\varphi,\psi \in \CHI_T^\mathrm{pw}$ satisfy~\eqref{eq:fullmildweak} with initial data $\varphi_0,\psi_0$.

For $t\in[0,T]$, we have
\begin{align*}
	\sum_{j=1}^{M} \Norm[C^{0,\gamma}_\mathrm{par}(\overline{\Omega}^j_t)]{\varphi - \psi}
	\leq \sum_{j=1}^{M} &\Norm[C^{0,\gamma}(\overline{\Omega}^j)]{\varphi_0 - \psi_0} + Q \Norm[C^{0,\gamma}_\mathrm{par}(\overline{\Omega}^j_t)]{\mathcal{P}(\varphi)(t,\cdot) - \mathcal{P}(\psi)(t,\cdot)}\\
	&+ \Norm[C^{0,\gamma}_\mathrm{par}(\overline{\Omega}^j_t)]{\int_{0}^{t} \beta(\varphi(s,\cdot)) \, \kappa(\mathcal{P}(\varphi)(s,\cdot)) - \beta(\psi(s,\cdot)) \, \kappa(\mathcal{P}(\psi)(s,\cdot)) \D s}\\
	\leq \sum_{j=1}^{M} &\Norm[C^{0,\gamma}(\overline{\Omega}^j)]{\varphi_0 - \psi_0} + (C_1 t^{\gamma/2} + C_2 t^{1-\gamma/2}) \Norm[C^{0,\gamma}_\mathrm{par}(\overline{\Omega}^j_t)]{\varphi - \psi},
\end{align*}
where the estimate of the second term follows as in \eqref{eq:pwlipschitzparabolic} 
with $C_1$ from~\eqref{eq:pwlipschitzparabolic}, independent of $\varphi,\psi$, and the estimate of the third term is obtained as in \eqref{eq:parabolic_lipschitz_help}. This yields
\begin{align*}
	\sum_{j=1}^{M} \Norm[C^{0,\gamma}_\mathrm{par}(\overline{\Omega}^j_t)]{\varphi - \psi}
	\leq \frac{1}{1 - (C_1 t^{\gamma/2} + C_2 t^{1-\gamma/2})} \sum_{j=1}^{M} \Norm[C^{0,\gamma}(\overline{\Omega}^j)]{\varphi_0 - \psi_0},
\end{align*}
where one should note that we chose $T$ such that $C_1 t^{\gamma/2} + C_2 t^{1-\gamma/2} < 1$ for all $t \leq T$ in Proposition~\ref{prp:contr}. Hence we obtain
\begin{align*}
	\sum_{j=1}^{M} \Norm[C^{0,\gamma}_\mathrm{par}(\overline{\Omega}^j_t)]{\varphi - \psi}
	\leq C \sum_{j=1}^{M} \Norm[C^{0,\gamma}(\overline{\Omega}^j)]{\varphi_0 - \psi_0}
\end{align*}
with a $C>0$.

\REV{To show (ii),} suppose that $T_{\mathrm{max}}<\infty$ with
\begin{align*}
	\sup_{t\in[0, T_{\text{max}})} \bl \sum_{j=1}^{M} \Norm[C^{0,\gamma}(\overline{\Omega}^j)]{\varphi(t,\cdot)} + \REV{\Norm[L^\infty(\Omega)]{\frac{1}{\varphi(T,\cdot)}}} \br \leq K 
\end{align*}
for some $K>0$.
By Theorem~\ref{th:parapiecewholder} and Proposition~\ref{prp:contr}, $\varphi(T_{\text{max}},\cdot) \in C^{0,\gamma}(\overline{\Omega}^j)$ and $u(T_{\text{max}},\cdot) \in C^{1,\gamma}(\overline{\Omega}^j)$ for $j = 1,\ldots,M$. Hence our local existence theory for $\tilde{\varphi}_0 = \varphi(T_{\text{max}},\cdot)$, $\tilde{u}_0 = u(T_{\text{max}},\cdot)$, $\tilde{R}=K+1$ and $\tilde{\epsilon}=\tfrac{1}{2}K^{-1}$ yields the existence of a solution on $[0,T_{\text{max}}+\delta]$ for some $\delta > 0$, contradicting the maximality of $T_{\text{max}}$.
\end{proof}

\begin{rem}
	Using a result on piecewise H\"older regularity of solutions of elliptic problems obtained in~\cite{Li2000} analogous to Theorem \ref{th:parapiecewholder}, one can also deduce existence and uniqueness of solutions to the viscous problem \eqref{eq:generalviscousmildweak} with piecewise H\"older regularity for any spatial dimension~$d$.
\end{rem}

\subsection*{Acknowledgements}

The authors would like to thank Evangelos Moulas for introducing them to the models discussed in this work and for helpful discussions.

\bibliographystyle{plain}

\appendix

\section{Model derivation}\label{sec:deriv}

We give a brief overview of the derivation of~\eqref{eq:model}. We consider a two-phase flow model for fluid in a porous rock in Eulerian coordinates. Let $\rho^{\mathrm{s}}$ and $v^{\mathrm{s}}$ denote the density and velocity of the solid matrix, respectively, and let $\rho^{\mathrm{f}}$ and $v^{\mathrm{f}}$ denote density and velocity of the fluid. Moreover, let $\phi$ denote the porosity, that is, the volume fraction of the fluid.
Conservation of mass for both phases then yields
\begin{equation}\label{eq:balance}
	\partial_t \bigl(\rho^{\mathrm{f}} \phi \bigr)
	= -\nabla \cdot \bigl(\rho^{\mathrm{f}} \phi {v}^{\mathrm{f}}\bigr), 
	\qquad
	\partial_t \bigl(\rho^{\mathrm{s}} (1-\phi)\bigr)
	= -\nabla \cdot \bigl(\rho^{\mathrm{s}} (1-\phi) v^\mathrm{s}\bigr) \,.
\end{equation}
\REV{We now rewrite these  equations} in terms of the material derivatives
$D^{\mathrm{p}}_t = \partial_t + {v}^\mathrm{p} \cdot \nabla$ for $\mathrm{p} \in \{\mathrm{s},\mathrm{f}\}$. 
We assume both phases to be incompressible, that is, $D^{\mathrm{p}}_t \rho^{\mathrm{p}} = 0$ for $\mathrm{p} \in \{\mathrm{s},\mathrm{f}\}$.
For the solid phase, this leads to the equation
\begin{equation}\label{eq:solid}
	-\frac{1}{1-\phi} D^{\mathrm{s}}_t \phi + \nabla \cdot v^{\mathrm{s}}
	= 0\,.
\end{equation}
For the last term on the left in~\eqref{eq:solid} we now use a rheology that models viscoelastic behavior of the general form
\begin{equation}\label{eq:rheo0}
	\nabla \cdot v^\mathrm{s}
	= - \frac{1}{(1-\phi) K_{\phi}} D_t^\mathrm{s} u - \frac{1}{(1-\phi) \, \eta_{\phi}(u)} u\,.
\end{equation}
Here $u$ describes the effective pressure, $\eta_{\phi}(u)$ is the effective viscosity and $K_{\phi}$ is the effective compressibility.
Following \cite[eq.~(11)]{Yarushina2015a} in the incompressible limit, we obtain
\[ 
	K_{\phi} =  \frac{K}{1-\phi}
\]
with constant $K>0$. Moreover, following the approach of \cite{McKenzie1984} (see also \cite[eq.~(57)]{Yarushina2015a}) we use
\[
	\eta_\phi = \frac{\sigma(u)}{(1-\phi) \phi^m}    ,
\]
where $m$ is a constant viscosity exponent, typically $m=1$. Here  $\sigma$ incorporates effects of decompaction weakening as in~\cite{Raess2019}. With $b(\phi) = \phi^m$, we thus obtain
\begin{equation}\label{eq:rheo}
	\nabla \cdot v^\mathrm{s} = - \frac{1}{K} D_t^\mathrm{s} u - \frac{b(\phi)}{\sigma(u)} u\,.
\end{equation}
Lacking a similar rheology for the fluid, we rewrite \eqref{eq:balance} for the fluid phase in the form
\[
	D^\mathrm{s}_t \phi + \nabla \cdot \bigl(\phi (v^\mathrm{f} - v^\mathrm{s}) \bigr) + \phi \nabla \cdot v^\mathrm{s}
	= 0\,.
\]
To the last term on the left, we apply Darcy's law in the form
\[
	\phi (v^\mathrm{f} - v^\mathrm{s})
	= a(\phi) \bigl(\nabla u + (1-\phi) ( \rho^\mathrm{s} - \rho^\mathrm{f}) g\bigr),
\]
where the Carman-Kozeny relationship yields $a(\phi) = a_0 \phi^n$ as in \eqref{eq:CK}, usually with $n \in [2,3]$, and $g$ describes the effects of gravity. We thus obtain
\begin{align}\label{eq:fluid}
	D^\mathrm{s}_t \phi + \nabla \cdot \left( a(\phi) (\nabla u + (1-\phi) ( \rho^\mathrm{s} - \rho^\mathrm{f}) \, g) \right) + \phi \nabla \cdot v^\mathrm{s}
	= 0.
\end{align}
Our incompressibility assumption justifies assuming $f = ( \rho^\mathrm{s} - \rho^\mathrm{f}) {g}$ to be constant.
Moreover, we replace $D_t^\mathrm{s}$ by $ \partial_t$, since the solid matrix can be assumed to move at negligible speed compared to the fluid. 
Combining \eqref{eq:solid} with \eqref{eq:rheo} yields
\[
\partial_t \phi
= - (1-\phi) \left( \frac{1}{K} \partial_t u + \frac{b(\phi)}{\sigma(u)} u \right).
\]
Moreover, adding \eqref{eq:solid} and \eqref{eq:fluid} and again using \eqref{eq:rheo} yields
\[
\frac{1}{K} \partial_t u
= \nabla \cdot \bigl( a(\phi) (\nabla u + (1-\phi) f) \bigr) - \frac{b(\phi)}{\sigma(u)} u  .
\]
Setting $Q = 1/K$, we thus arrive at~\eqref{eq:model}.

\section{$C^k$-convergence of $\Delta_{\kappa,\uzero}$}\label{sec:Ckconv}
We aim to show that $\lim_{\dphi \to 0} \Norm[C^k(\overline{\Omega})]{\udphi - \uzero}$ implies $\lim_{\dphi \to 0} \Norm[C^k(\overline{\Omega})]{\Delta_{\kappa,\uzero}(\udphi) - \kappa'(\uzero)}	= 0$.
We may assume that $\udphi \neq \uzero$. 
For the first derivative, we obtain
\begin{multline}\label{eq:deriv}
	\partial_{x_i} \left( \frac{\kappa(\udphi) - \kappa(\uzero)}{\udphi - \uzero} - \kappa'(\uzero) \right)\\
	\begin{split}
		&= \frac{(\kappa'(\udphi) \, \partial_{x_i} \udphi - \kappa'(\uzero) \, \partial_{x_i} \uzero) (\udphi - \uzero) - (\kappa(\udphi) - \kappa(\uzero)) \, \partial_{x_i} (\udphi - \uzero)}{(\udphi - \uzero)^2} - \kappa''(\uzero) \partial_{x_i} \uzero\\
		&= F_1 \, F_2 - F_3 \, F_4
	\end{split}
\end{multline}
where
\begin{align*}
	F_1 &= \partial_{x_i} \uzero ,  &  F_2 &=  \frac{\kappa'(\udphi) - \kappa'(\uzero)}{\udphi - \uzero} - \kappa''(\uzero), \\
	F_3 &= \partial_{x_i} (\udphi - \uzero) , &  F_4 & = \frac{\kappa(\udphi) + \kappa'(\udphi) (\uzero - \udphi)  - \kappa(\uzero)}{(\udphi - \uzero)^2} ,
\end{align*}
and observe that $F_2, F_3$ go to zero and $F_1, F_4$ are bounded if $\dphi \to 0$. This can be extended to higher derivatives by using a generalized product rule. It is clear that the higher derivatives of $F_2$ and $F_3$ still tend to zero (for $F_2$ we just repeat~\eqref{eq:deriv} inductively) and the ones of $F_1$ are still bounded. For the derivatives of $F_4$ we have
\begin{multline*}
	\partial_{x_j} \left( \frac{1}{(\udphi - \uzero)^{n+1}} \left( \kappa(\uzero) - \sum_{k=0}^n \frac{\kappa^{(k)}(\udphi)}{k!} (\uzero - \udphi)^k \right) \right)\\
	\begin{split}
		=& - \frac{n+1}{(\udphi - \uzero)^{n+2}} \, \partial_{x_j} (\udphi - \uzero) \left( \kappa(\uzero) - \sum_{k=0}^n \frac{\kappa^{(k)}(\udphi)}{k!} (\uzero - \udphi)^k \right)\\
		& + \frac{1}{(\udphi - \uzero)^{n+1}} \Bigg( \kappa'(\uzero) \, \partial_{x_j} \uzero - \sum_{k=0}^n \frac{\kappa^{(k+1)}(\udphi)}{k!} (\uzero - \udphi)^k \, \partial_{x_j} \udphi\\
		&\hspace*{3cm} - \sum_{k=1}^n \frac{\kappa^{(k)}(\udphi)}{(k-1)!} (\uzero - \udphi)^{k-1} \partial_{x_j} (\uzero - \udphi) \Bigg)\\
		=& - \frac{n+1}{(\udphi - \uzero)^{n+2}} \, \partial_{x_j} (\udphi - \uzero) \left( \kappa(\uzero) - \sum_{k=0}^{n+1} \frac{\kappa^{(k)}(\udphi)}{k!} (\uzero - \udphi)^k \right)\\
		& + \frac{1}{(\udphi - \uzero)^{n+1}} \Bigg( \kappa'(\uzero) \, \partial_{x_j} \uzero - \frac{\kappa^{(n+1)}(\udphi)}{n!} (\uzero - \udphi)^n \, \partial_{x_j} \uzero\\
		&\hspace*{3cm} - \sum_{k=0}^{n-1} \frac{\kappa^{(k+1)}(\udphi)}{k!} (\uzero - \udphi)^k\, \partial_{x_j} \udphi\\
		&\hspace*{3cm} - \sum_{k=1}^n \frac{\kappa^{(k)}(\udphi)}{(k-1)!} (\uzero - \udphi)^{k-1} \, \partial_{x_j} (\uzero - \udphi) \Bigg)\\
		=& - \frac{n+1}{(\udphi - \uzero)^{n+2}} \, \partial_{x_j} (\udphi - \uzero) \left( \kappa(\uzero) - \sum_{k=0}^{n+1} \frac{\kappa^{(k)}(\udphi)}{k!} (\uzero - \udphi)^k \right)\\
		& + \frac{1}{(\udphi - \uzero)^{n+1}} \, \partial_{x_j} \uzero \left( \kappa'(\uzero) - \sum_{k=0}^n \frac{\kappa^{(k+1)}(\udphi)}{k!} (\uzero - \udphi)^k \right),
	\end{split}
\end{multline*}
which yields the desired result by induction.

\section{Well-posedness of the semilinear parabolic problem}\label{sec:appexistence}

\begin{proof}[Proof of Theorem~\ref{th:exparabolic}]
	The proof follows the standard Galerkin method for parabolic PDEs. 
	Let $N\in \N$ be fixed. We consider the finite-dimensional space $E_N=\operatorname{span}\{w_1,\ldots,w_N\}$ consisting of the first $N$ vectors of an orthonormal basis $\{w_k\colon k\in \N\}$ of $L^2(\Omega)$ where we assume that $\{w_k\colon k\in \N\}$ is also an orthogonal basis of $W^{1,2}_0(\Omega)$. %
	Let $T>0$ be given and let $u_N\colon [0,T]\to W^{1,2}_0(\Omega)$ be of the form
	\begin{align*}
		u_N(t)=\sum_{k=1}^N d_N^k(t) w_k,
	\end{align*}
	where $d_N^k\colon [0,T]\to \R$ are absolutely continuous scalar coefficient functions. We set
	\begin{align}\label{eq:dinit}
		d_N^k(0)= \int_\Omega u_0 w_k\di x,\quad k=1,\ldots,N,
	\end{align}
	implying that 
	\begin{align}\label{eq:u0init}
		u_N(0)=\sum_{k=1}^N \bl \int_\Omega u_0 w_k\di x \br w_k \in E_N,
	\end{align}
	i.e.\ $u_N(0)$ is the $L^2$-orthogonal projection of $u_0$ on $E_N$.
	We want to determine $\{d_N^1,\ldots,d_N^N\}$ such that $u_N$ satisfies the weak formulation~\eqref{eq:weakpara} for all test functions $\psi\in E_N$. Hence, we require that $\{d_N^1,\ldots,d_N^N\}$ satisfy
	\begin{multline}\label{eq:weak_para_approx}
		\int_\Omega \partial_t u_N(t) w_k \di x + \int_\Omega \Qalpha(\varphi(t)) \, \nabla u_N(t) \cdot \nabla w_k\di x + \int_\Omega \frac{\Qbeta(\varphi(t))}{\sigma(u_N(t))} u_N(t) w_k \di x\\
		= \int_\Omega w_k \nabla \cdot \Qalpha(\varphi(t)) \, \zeta(\varphi(t))\di x,\quad k=1,\ldots,N,\quad \text{a.e.\ } t\in [0,T],
	\end{multline}
	which is equivalent to the system of ODEs
	 \begin{align*}%
		 & \partial_t d_N^k(t) +  \sum_{j=1}^N d_N^j(t)  A^{jk}(t)+  \sum_{j=1}^N d_N^j(t)  B^{jk}(t, d_N(t))=g^k(t),\quad k=1,\ldots,N,\quad \text{a.e.\ } t\in [0,T],
	\end{align*}
	where 
	\begin{align*}
		d_N&=(d_1^N,\ldots,d_N^N), & 		A^{jk}(t)&=\int_\Omega \Qalpha(\varphi(t)) \, \nabla w_j \cdot \nabla w_k\di x,\\
		B^{jk}(t,d_N(t))&=\int_\Omega\frac{\Qbeta(\varphi(t))}{\sigma(\sum_{i=1}^N d_N^i(t) w_i)} w_j w_k\di x,  & 
		g^k(t)&=\int_\Omega w_k \nabla \cdot \Qalpha(\varphi(t)) \, \zeta(\varphi(t)) \di x,
	\end{align*}
	for $j,k=1,\ldots,N$.
	Setting $g=(g^1,\ldots,g^N)$, this yields
	\begin{align}\label{eq:odematrix}
		\partial_t d_N(t) +  A(t) d_N(t) +  B(t, d_N(t))d_N(t)=g(t),\quad \text{a.e.\ } t\in [0,T].
	\end{align}
	We have $A\in L^\infty(0,T;\R^{N\times N})$ since $\varphi\in L^\infty((0,T)\times\Omega)$. Further, $B(\cdot,d_N)d_N$ is Lipschitz continuous with respect to $d_N$, which follows from Assumption~\ref{ass:sigma} for $\sigma$. We can write~\eqref{eq:odematrix} as
	\begin{align}\label{eq:fixedpointparabolic}
		d_N=\Phi(d_N)
	\end{align}
	with 
	\begin{align*}
		\Phi(d_N)(t)=d_N(0) -\int_0^t A(s) d_N(s) \di s -\int_0^t B(s,d_N(s)) d_N(s) \di s+\int_0^t g(s)\di s
	\end{align*}
	and initial data $d_N(0)$ in~\eqref{eq:dinit}.
	This implies that $\Phi\colon C([0,T_\ast];\R^N)\to C([0,T_\ast];\R^N)$ for any $0<T_\ast \leq T$. Further, for any $d,\tilde{d} \in C([0,T_\ast];\R^N)$, we have
	\begin{align*}
		\| \Phi(d)-\Phi(\tilde d)\|_{C([0,T_\ast];\R^N)}\leq MT_\ast \| d-\tilde d\|_{C([0,T_\ast];\R^N)},
	\end{align*}
	where $M>0$ depends on $\| A\|_{L^\infty(0,T;\R^{N\times N})}$ and the Lipschitz constant of $d_N\mapsto B(\cdot,d_N)d_N$.
	Choosing $T_\ast$ such that $MT_\ast <1$ implies that the map $\Phi$ is a contraction on $C([0,T_\ast];\R^N)$. By the contraction mapping theorem, there exists a unique solution $d_N$ to~\eqref{eq:fixedpointparabolic} on $[0,T_\ast]$ and by applying this result for a finite number of times we obtain a solution $d_N\in C([0,T];\R^N)$.
	We conclude that for any $N\in \N$ and any $T>0$, there exists a unique  $u_N\colon [0,T]\to E_N$ with initial data~\eqref{eq:u0init} such that~\eqref{eq:weak_para_approx} is satisfied.
	
	For an a priori energy estimate for $u_N$, we consider the weak formulation ~\eqref{eq:weakpara} for $u_N$ with test function $u_N(t)\in E_N$ for any $t\in [0,T]$ and we have
	\begin{multline*}
		\frac{1}{2}\frac{\di }{\di t}\int_\Omega  u_N^2(t) \di x + \int_\Omega \Qalpha(\varphi(t)) \, \nabla u_N(t) \cdot \nabla u_N(t)\di x +\int_\Omega \frac{\Qbeta(\varphi(t))}{\sigma(u_N(t))} u_N^2(t) \di x \\
		=\int_\Omega u_N(t) \, \nabla \cdot \Qalpha(\varphi(t)) \, \zeta(\varphi(t)) \di x,
	\end{multline*}
	implying that
	there exists a constant $C_1>0$ depending on $Q$, $\varphi$, $c_0$,  such that 
	\begin{align*}
		\frac{1}{2}\frac{\di }{\di t}\| u_N(t) \|_{L^2(\Omega)}^2 +C_1\|  u_N(t) \|_{W^{1,2}_0(\Omega)}^2
		\leq\int_\Omega u_N(t) \nabla \cdot \Qalpha(\varphi(t)) \, \zeta(\varphi(t)) \di x.
	\end{align*}
	Integration over $[0,T_\ast]$ for any $0<T_\ast\leq T$ yields
	\begin{multline*}
		\frac{1}{2}\| u_N(T_\ast) \|_{L^2(\Omega)}^2 +C_1\int_0^{T_\ast}\|  u_N(t) \|_{W^{1,2}_0(\Omega)}^2\ \di t \\
		\leq \int_0^{T_\ast}\int_\Omega u_N(t) \nabla \cdot \Qalpha(\varphi(t)) \, \zeta(\varphi(t)) \di x\di t+ \| u_0 \|_{L^2(\Omega)}^2,
	\end{multline*}
	where we used that  $\| u_N(0) \|_{L^2(\Omega)}^2=\sum_{k=1}^N |d_N^k(0)|^2=\sum_{k=1}^N |\int_\Omega u_0 w_k\di x|^2\leq \| u_0 \|_{L^2(\Omega)}^2$ by  Bessel's inequality.
	Since
	\begin{align*}
		&\int_0^{T_\ast}\int_\Omega u_N(t) \nabla \cdot \Qalpha(\varphi(t)) \, \zeta(\varphi(t))\di x\di t\\
		&\leq \int_0^{T_\ast}\| u_N(t)\|_{W^{1,2}_0(\Omega)} \| \Qalpha(\varphi(t)) \, \zeta(\varphi(t))\|_{L^2(\Omega)}\di t\\
		&\leq \frac{1}{2C_1}\int_0^{T_\ast} \| \Qalpha(\varphi(t)) \, \zeta(\varphi(t))\|_{L^2(\Omega)}^2 \di t+\frac{C_1}{2}\int_0^{T_\ast}\|  u_N(t) \|_{W^{1,2}_0(\Omega)}^2\ \di t,
	\end{align*}
	we obtain
	\begin{align*}
	   & \frac{1}{2}\| u_N(T_\ast) \|_{L^2(\Omega)}^2 +\frac{C_1}{2}\int_0^{T_\ast}\|  u_N(t) \|_{W^{1,2}_0(\Omega)}^2\ \di t \\
		&\leq \frac{1}{2C_1}\int_0^{T} \| \Qalpha(\varphi(t)) \, \zeta(\varphi(t))\|_{L^2(\Omega)}^2 \di t + \| u_0\|_{L^2(\Omega)}^2\\
		&\leq \max\left\{1,\frac{1}{2C_1}\right\} \bl \| \Qalpha(\varphi) \, \zeta(\varphi)\|_{L^2(0,T;L^2(\Omega))}^2 + \| u_0 \|_{L^2(\Omega)}^2 \br,
	\end{align*}
	implying that
	\begin{align*}
	  &\frac{1}{2} \sup_{t\in[0,T]}\| u_N(t) \|_{L^2(\Omega)}^2 +\frac{C_1}{2}\int_0^T\|  u_N(t) \|_{W^{1,2}_0(\Omega)}^2\ \di t\\
		&\leq 2 \max\left\{1,\frac{1}{2C_1}\right\} \bl \| \Qalpha(\varphi) \, \zeta(\varphi)\|_{L^2(0,T;L^2(\Omega))}^2 +  \| u_0 \|_{L^2(\Omega)}^2 \br .
	\end{align*}  
	We conclude that there exists $C_3>0$ such that
	\begin{align}\label{eq:energypara_help}
		\| u_N \|_{L^\infty(0,T,L^2(\Omega))}^2 +\| u_N \|_{L^2(0,T;W^{1,2}_0(\Omega))}^2
		\leq C_3 \bl \| \Qalpha(\varphi) \, \zeta(\varphi)\|_{L^2(0,T;L^2(\Omega))}^2 +  \| u_0 \|_{L^2(\Omega)}^2 \br.
	\end{align}  
	Next, we estimate $\partial_t u_N$. Note that $\partial_t u_N(t)\in E_N$ implying $\int_\Omega \partial_t u_N(t) w_k \di x=0$ for all $k\in\N$ with $k>N$. Since $u_N$ satisfies~\eqref{eq:weak_para_approx}, this yields 
	\begin{align*}
		\int_\Omega \partial_t u_N(t) \psi \di x
		&\leq C_4 \bl \| u_N(t)\|_{W_0^{1,2}(\Omega)} \| \psi\|_{W_0^{1,2}(\Omega)} +\int_\Omega \nabla \psi  \cdot \Qalpha(\varphi(t)) \, \zeta(\varphi(t))\di x \br \\
		&\leq C_4 \| \psi\|_{W_0^{1,2}(\Omega)} \bl \| u_N(t)\|_{W_0^{1,2}(\Omega)}  +\| \Qalpha(\varphi(t)) \, \zeta(\varphi(t))\|_{L^2(\Omega)} \br
	\end{align*}
	for any $\psi \in W^{1,2}_0(\Omega)$, where $C_4>0$ depends on $Q$, $\varphi$ and $\sigma$. This implies that
	\begin{align*}
		\| \partial_t u_N(t)\|_{W^{-1,2}(\Omega)}
		\leq C_4 \bl \| u_N(t)\|_{W_0^{1,2}(\Omega)}  +\| \Qalpha(\varphi(t)) \, \zeta(\varphi) \|_{L^2(\Omega)} \br
	\end{align*}
	and hence there exists $C_5>0$ such that 
	\begin{align*}
		\| \partial_t u_N\|_{L^2(0,T;W^{-1,2}(\Omega))}^2
		&=\int_0^T\| \partial_t u_N(t)\|_{W^{-1,2}(\Omega)}^2\di t \\
		&\leq C_5 \bl \int_0^T\| u_N(t)\|_{W_0^{1,2}(\Omega)}^2 \di t + \int_0^T \| \Qalpha(\varphi(t)) \, \zeta(\varphi(t))\|_{L^2(\Omega)}^2\di t \br \\
		& \leq C_5 (C_3+1) \bl \| \Qalpha(\varphi) \, \zeta(\varphi)\|_{L^2(0,T;L^2(\Omega))}^2 +  \| u_0 \|_{L^2(\Omega)}^2 \br ,
	\end{align*}
	where the last inequality follows from~\eqref{eq:energypara_help}. Together with~\eqref{eq:energypara_help}, we conclude that there exists $C>0$ such that
	\begin{multline*}
		\| u_N\|_{L^\infty(0,T;L^2(\Omega))}^2 + \| u_N\|_{L^2(0,T;W^{1,2}_0(\Omega))}^2 + \| \partial_t u_N \|_{L^2(0,T;W^{-1,2}(\Omega))}^2\\
		\leq C \bl \| \Qalpha(\varphi) \, \zeta(\varphi)\|_{L^2(0,T;L^{2}(\Omega))}^2 +\|u_0\|_{L^2(\Omega)}^2 \br .
	\end{multline*}
	and~\eqref{eq:bound_un} immediately follows for $u_N$.
	
	Due to the uniform bounds in $N$ in~\eqref{eq:bound_un}, $u_N$ is bounded in $C(\REV{[0,T]};L^2(\Omega))\cap L^2(0,T;W^{1,2}_0(\Omega))$ and $\partial_t u_N$ is bounded in $L^2(0,T;W^{-1,2}(\Omega))$. By reflexivity of the respective Hilbert spaces, there exist $u\in C(\REV{[0,T]};L^2(\Omega))\cap L^2(0,T;W^{1,2}_0(\Omega))$ and $\partial_t u\in L^2(0,T;W^{-1,2}(\Omega))$, and subsequences of $u_N$, $\partial_t u_N$, again denoted by $u_N$, $\partial_t u_N$, respectively, such that $u_N\rightharpoonup u$ in $L^2(0,T;W^{1,2}_0(\Omega))$ and $\partial_t u_N\rightharpoonup \partial_t u\in L^2(0,T;W^{-1,2}(\Omega))$.  Next, we show that 
	\[
		\kappa(u_N) \rightarrow \kappa(u)
		\quad\text{in $L^2(0,T;L^2(\Omega))$.}
	\]
	Since  $u_N \rightharpoonup u$ in $L^2(0,T;W^{1,2}_0(\Omega))$ and $ W_0^{1,2}(\Omega)$ is compactly embedded into $L^2(\Omega)$, we obtain that there exists a subsequence of $u_N$ (not relabelled) such that $u_N\to u$ in $L^2(0,T;L^2(\Omega))$ by the Aubin–Lions lemma. This yields
	\begin{align*}
		\Norm[L^2(0,T;L^2(\Omega))]{\kappa(u_N) - \kappa(u)}^2
		&= \int_0^T \int_\Omega \bl \kappa(u_N) - \kappa(u) \br ^2 \di x \di t\\
		&\leq c_L^2 \int_0^T \int_\Omega \bl u_N - u \br ^2 \di x \di t\\
		&= c_L^2 \Norm[L^2(0,T;L^2(\Omega))]{u_N - u}^2
	\end{align*}
	making use of the global Lipschitz property of
	$\kappa$.
	
	Note that we only have weak convergence of the gradient, which does \textit{not} imply a.e.\ convergence, but since the limit $\nabla u$ fulfills the space-time integrated version of~\eqref{eq:weakpara} for all space-time test functions we obtain that
	~\eqref{eq:weakpara} is satisfied for a.e.\ $t\in[0,T]$.
	From \cite[Chapter 5.9, Theorem 3]{Evans2010}, it follows that  $u\in C([0,T]; L^2(\Omega))$.
	
	To show that $u(0)=u_0$, note that from~\eqref{eq:weakpara} we obtain that
	\begin{multline*}
		-\int_0^T\int_\Omega u \partial_t \psi \di x \di t + \int_0^T \int_\Omega \Qalpha(\varphi(t)) \, \nabla u(t) \cdot \nabla \psi\di x\di t +\int_0^T\int_\Omega \Qbeta(\varphi(t)) \, \kappa(u(t))\psi \di x\di t\\
		=\int_0^T\int_\Omega \psi \nabla \cdot \Qalpha(\varphi(t)) \, \zeta(\varphi(t)) \di x\di t + \left< u(0), \psi(0) \right>_{L^2(\Omega)}
	\end{multline*}
	for all $\psi\in C^1([0,T];W^{1,2}_0(\Omega))$ with $\psi(T)=0$. Replacing $u$ by $u_N$ and considering the limit $N\to \infty$ yields
	\begin{multline*}
		-\int_0^T\int_\Omega u \partial_t \psi \di x \di t + \int_0^T \int_\Omega \Qalpha(\varphi(t)) \, \nabla u(t) \cdot \nabla \psi\di x\di t +\int_0^T\int_\Omega \Qbeta(\varphi(t)) \, \kappa(u(t))\psi \di x\di t\\
		=\int_0^T\int_\Omega \psi \nabla \cdot \Qalpha(\varphi(t)) \, \zeta(\varphi(t)) \di x\di t+ \left< u_0, \psi(0) \right>_{L^2(\Omega)}
	\end{multline*}
	since $u_N(0)\to u_0$ in $L^2(\Omega)$.
	As $\psi(0)$ is arbitrary, we conclude that $u(0)=u_0$. Hence, $u$ is a weak solution to~\eqref{eq:weakpara}.
	
	To show the uniqueness of solutions, we consider two weak solutions $u,v\in C([0,T]; L^2(\Omega))\cap L^2(0,T; W^{1,2}_0(\Omega))$ satisfying~\eqref{eq:weakpara} with $u(0)=v(0)=u_0$, implying 
	\begin{multline*}
		\int_\Omega \partial_t (u(t)-v(t)) \psi \di x + \int_\Omega \Qalpha(\varphi(t)) \, \nabla (u(t)-v(t)) \cdot \nabla \psi\di x \\
		+ \int_\Omega \Qbeta(\varphi(t)) \bl \kappa(u(t)) - \kappa(v(t)) \br \psi \di x =0.
	\end{multline*}
	We set $\psi(t)=u(t)-v(t)\in W^{1,2}_0(\Omega)$ and obtain
	\begin{multline*}
		\frac{1}{2}\frac{\di}{\di t}\Norm[L^2(\Omega)]{u(t)-v(t)}^2 + c_P \| u(t)-v(t))\|_{L^2(\Omega)}^2\\
		\begin{split}
		&\leq \int_\Omega \partial_t (u(t)-v(t)) (u(t)-v(t)) \di x + \int_\Omega \Qalpha(\varphi(t)) \AV{\nabla (u(t)-v(t))}^2 \di x\\
		&=-\int_\Omega \Qbeta(\varphi(t)) \bl \kappa(u(t)) - \kappa(v(t)) \br (u(t)-v(t)) \di x\\
		&\leq c_L\Norm[L^\infty(\Omega)]{\Qbeta(\varphi(t))} \Norm[L^2(\Omega)]{u(t)-v(t)}^2,
		\end{split}
	\end{multline*}
	where $c_L$ denotes the Lipschitz constant of
	$\kappa$ from Assumptions~\ref{ass:sigma} and $c_P>0$ is the Poincaré constant.
	As $\|u(0)-v(0)\|_{L^2(\Omega)}=0$, Grönwall's inequality implies that $\|u(t)-v(t)\|_{L^2(\Omega)}=0$ for all $t\geq 0$, so $u(t)=v(t)$ for a.e.\ $x\in\Omega$ and shows the uniqueness of the weak solution of~\eqref{eq:weakpara}.
	\end{proof}

\section{Parabolic norm equivalences}\label{sec:equiv}
We next establish uniform proportionality of $\Norm[C^{0,\gamma}_\mathrm{par}(\overline{\Omega}^j_T)]{f} $ and $|f|_{\gamma/2,\gamma}$ as well as $\Norm[C^{1,\gamma}_\mathrm{par}(\overline{\Omega}^j_T)]{f}$ and $|f|_{(1+\gamma)/2,1+\gamma}$ with $\Omega^j$ as in \eqref{eq:partition} for $j=1,\ldots,M$, where the former norms are defined in~\eqref{eq:holder} and the latter, used in~\cite{Dong2021}, are defined as 
\begin{align*}
	|f|_{\gamma/2,\gamma}
	&= \sup_{\overline{\Omega}^j_T} |f(t,x)| + \sup_{\substack{t_1,t_2,x_1,x_2\\(t_1,x_1) \neq (t_2,x_2)}} \frac{|f(t_1,x_1) - f(t_2,x_2)|}{|x_1 - x_2|^\gamma + |t_1 - t_2|^{\gamma/2}}
\end{align*}	
and
\begin{align*}
	|f|_{(1+\gamma)/2,1+\gamma}
	=& \sup_{\overline{\Omega}^j_T} |f(t,x)| + \sup_{\substack{t_1,t_2,x_1,x_2\\(t_1,x_1) \neq (t_2,x_2)}} \frac{|\nabla f(t_1,x_1) - \nabla f(t_2,x_2)|}{|x_1 - x_2|^\gamma + |t_1 - t_2|^{\gamma/2}} + \sup_{\substack{t_1,t_2,x\\t_1 \neq t_2}} \frac{|f(t_1,x) - f(t_2,x)|}{|t_1 - t_2|^{(1+\gamma)/2}},
\end{align*}	
respectively. Part of the proofs relies on the assumption that $f(0,\cdot) \equiv 0$, which we assume to be given for now.

We start with showing the first equivalence. We have
\begin{align*}
	\Norm[C^{0,\gamma}_\mathrm{par}(\overline{\Omega}^j_T)]{f}
	&= \sup_{\overline{\Omega}^j_T} |f(t,x)| + \sup_{\substack{t,x_1,x_2\\x_1 \neq x_2}} \frac{|f(t,x_1) - f(t,x_2)|}{|x_1 - x_2|^\gamma} + \sup_{\substack{t_1,t_2,x\\t_1 \neq t_2}} \frac{|f(t_1,x) - f(t_2,x)|}{|t_1 - t_2|^{\gamma/2}}\\
	&\leq \sup_{\overline{\Omega}^j_T} |f(t,x)| + 2 \sup_{\substack{t_1,t_2,x_1,x_2\\(t_1,x_1) \neq (t_2,x_2)}} \frac{|f(t_1,x_1) - f(t_2,x_2)|}{|x_1 - x_2|^\gamma + |t_1 - t_2|^{\gamma/2}}\leq 2 |f|_{\gamma/2,\gamma}
\end{align*}
and
\begin{align*}
	|f|_{\gamma/2,\gamma}
	&= \sup_{\overline{\Omega}^j_T} |f(t,x)| + \sup_{\substack{t_1,t_2,x_1,x_2\\(t_1,x_1) \neq (t_2,x_2)}} \frac{|f(t_1,x_1) - f(t_2,x_2)|}{|x_1 - x_2|^\gamma + |t_1 - t_2|^{\gamma/2}}\\
	&\leq \sup_{\overline{\Omega}^j_T} |f(t,x)| + \sup_{\substack{t_1,t_2,x_1,x_2\\(t_1,x_1) \neq (t_2,x_2)}} \frac{|f(t_1,x_1) - f(t_1,x_2)|}{|x_1 - x_2|^\gamma + |t_1 - t_2|^{\gamma/2}} + \sup_{\substack{t_1,t_2,x_1,x_2\\(t_1,x_1) \neq (t_2,x_2)}} \frac{|f(t_1,x_2) - f(t_2,x_2)|}{|x_1 - x_2|^\gamma + |t_1 - t_2|^{\gamma/2}}\\
	&\leq \sup_{\overline{\Omega}^j_T} |f(t,x)| + \sup_{\substack{t,x_1,x_2\\x_1 \neq x_2}} \frac{|f(t,x_1) - f(t,x_2)|}{|x_1 - x_2|^\gamma} + \sup_{\substack{t_1,t_2,x\\t_1 \neq t_2}} \frac{|f(t_1,x) - f(t_2,x)|}{|t_1 - t_2|^{\gamma/2}}
	= \Norm[C^{0,\gamma}_\mathrm{par}(\overline{\Omega}^j_T)]{f}.
\end{align*}
Next we show the second equivalence using that $\nabla f(0,\cdot) \equiv 0$. We have
\begin{align*}
	\Norm[C^{1,\gamma}_\mathrm{par}(\overline{\Omega}^j_T)]{f}
	=& \sup_{\overline{\Omega}^j_T} |f(t,x)| + \sup_{\overline{\Omega}^j_T} |\nabla f(t,x)| + \sup_{\substack{t,x_1,x_2\\x_1 \neq x_2}} \frac{|\nabla f(t,x_1) - \nabla f(t,x_2)|}{|x_1 - x_2|^\gamma}\\
	&+ \sup_{\substack{t_1,t_2,x\\t_1 \neq t_2}} \frac{|\nabla f(t_1,x) - \nabla f(t_2,x)|}{|t_1 - t_2|^{\gamma/2}} + \sup_{\substack{t_1,t_2,x\\t_1 \neq t_2}} \frac{|f(t_1,x) - f(t_2,x)|}{|t_1 - t_2|^{(1+\gamma)/2}}\\
	\leq& \sup_{\overline{\Omega}^j_T} |f(t,x)| + T^{\gamma/2} \sup_{\overline{\Omega}^j_T} \frac{|\nabla f(t,x) - \nabla f(0,x)|}{|t - 0|^{\gamma/2}}\\
	&+ 2 \sup_{\substack{t_1,t_2,x_1,x_2\\(t_1,x_1) \neq (t_2,x_2)}} \frac{|\nabla f(t_1,x_1) - \nabla f(t_2,x_2)|}{|x_1 - x_2|^\gamma + |t_1 - t_2|^{\gamma/2}} + \sup_{\substack{t_1,t_2,x\\t_1 \neq t_2}} \frac{|f(t_1,x) - f(t_2,x)|}{|t_1 - t_2|^{(1+\gamma)/2}}\\
	\leq& \sup_{\overline{\Omega}^j_T} |f(t,x)| + T^{\gamma/2} \sup_{\substack{t_1,t_2,x\\t_1 \neq t_2}} \frac{|\nabla f(t_1,x) - \nabla f(t_2,x)|}{|t_1 - t_2|^{\gamma/2}}\\
	&+ 2 \sup_{\substack{t_1,t_2,x_1,x_2\\(t_1,x_1) \neq (t_2,x_2)}} \frac{|\nabla f(t_1,x_1) - \nabla f(t_2,x_2)|}{|x_1 - x_2|^\gamma + |t_1 - t_2|^{\gamma/2}} + \sup_{\substack{t_1,t_2,x\\t_1 \neq t_2}} \frac{|f(t_1,x) - f(t_2,x)|}{|t_1 - t_2|^{(1+\gamma)/2}}\\
	\leq& (2 + T^{\gamma/2}) |f|_{(1+\gamma)/2,1+\gamma}
\end{align*}
and
\begin{align*}
	|f|_{(1+\gamma)/2,1+\gamma}
	\leq& \sup_{\overline{\Omega}^j_T} |f(t,x)| + \sup_{\substack{t_1,t_2,x_1,x_2\\(t_1,x_1) \neq (t_2,x_2)}} \frac{|\nabla f(t_1,x_1) - \nabla f(t_1,x_2)|}{|x_1 - x_2|^\gamma + |t_1 - t_2|^{\gamma/2}}\\
	&+ \sup_{\substack{t_1,t_2,x_1,x_2\\(t_1,x_1) \neq (t_2,x_2)}} \frac{|\nabla f(t_1,x_2) - \nabla f(t_2,x_2)|}{|x_1 - x_2|^\gamma + |t_1 - t_2|^{\gamma/2}} + \sup_{\substack{t_1,t_2,x\\t_1 \neq t_2}} \frac{|f(t_1,x) - f(t_2,x)|}{|t_1 - t_2|^{(1+\gamma)/2}}\\
	\leq& \sup_{\overline{\Omega}^j_T} |f(t,x)| + \sup_{\overline{\Omega}^j_T} |\nabla f(t,x)| + \sup_{\substack{t,x_1,x_2\\x_1 \neq x_2}} \frac{|\nabla f(t,x_1) - \nabla f(t,x_2)|}{|x_1 - x_2|^\gamma}\\
	&+ \sup_{\substack{t_1,t_2,x\\t_1 \neq t_2}} \frac{|\nabla f(t_1,x) - \nabla f(t_2,x)|}{|t_1 - t_2|^{\gamma/2}} + \sup_{\substack{t_1,t_2,x\\t_1 \neq t_2}} \frac{|f(t_1,x) - f(t_2,x)|}{|t_1 - t_2|^{(1+\gamma)/2}}\\
	=& \Norm[C^{1,\gamma}_\mathrm{par}(\overline{\Omega}^j_T)]{f},
\end{align*}
which concludes the proof.

\section{Parabolic H\"older norm estimates}\label{sec:est}
We show auxiliary estimates for two parabolic Hölder norms. We first observe, for $f$ with $f(0,\cdot) \equiv 0$,
\begin{align}\label{eq:paraest}
	\begin{split}
		\Norm[C^{0,\gamma}_\mathrm{par}(\overline{\Omega}^j_T)]{f}
		=& \sup_{\overline{\Omega}^j_T} |f(t,x)| + \sup_{\substack{t,x_1,x_2\\x_1 \neq x_2}} \frac{|f(t,x_1) - f(t,x_2)|}{|x_1 - x_2|^\gamma} + \sup_{\substack{t_1,t_2,x\\t_1 \neq t_2}} \frac{|f(t_1,x) - f(t_2,x)|}{|t_1 - t_2|^{\gamma/2}}\\
		\leq& \sup_{\overline{\Omega}^j_T} |f(t,x)| + C \sup_{\overline{\Omega}^j_T} |\nabla f(t,x)| + \sup_{\substack{t_1,t_2,x\\t_1 \neq t_2}} \frac{|f(t_1,x) - f(t_2,x)|}{|t_1 - t_2|^{\gamma/2}}\\
		\leq& T^{\gamma/2} \Bigg( \sup_{\overline{\Omega}^j_T} \frac{|f(t,x) - f(0,x)|}{|t-0|^{\gamma/2}} + C \sup_{\overline{\Omega}^j_T} \frac{|\nabla f(t,x) - \nabla f(0,x)|}{|t-0|^{\gamma/2}}\\
		&\qquad + \sup_{\substack{t_1,t_2,x\\t_1 \neq t_2}} \frac{|f(t_1,x) - f(t_2,x)|}{|t_1 - t_2|^{\gamma}} \Bigg)\\
		\leq& T^{\gamma/2} \Bigg( C \sup_{\substack{t_1,t_2,x_1,x_2\\(t_1,x_1) \neq (t_2,x_2)}} \frac{|\nabla f(t_1,x_1) - \nabla f(t_2,x_2)|}{|x_1 - x_2|^\gamma + |t_1 - t_2|^{\gamma/2}} + 2 \sup_{\substack{t_1,t_2,x\\t_1 \neq t_2}} \frac{|f(t_1,x) - f(t_2,x)|}{|t_1 - t_2|^{(1+\gamma)/2}} \Bigg)\\
		\leq& \max\{2, C\} \, T^{\gamma/2} |f|_{(1+\gamma)/2,1+\gamma},
	\end{split}
\end{align}
where $C$ depends on the diameter of $\Omega$.
Note that we  can replace the norms on the left- and right-hand sides by their equivalent versions in Appendix \ref{sec:equiv}.

Next we derive an estimate for the time integral of a function $f \in C^{0,\gamma}_\mathrm{par}(\overline{\Omega}^j_T)$. 
We have
\begin{align*}
	\begin{split}
		\Norm[C^{0,\gamma}_\mathrm{par}(\overline{\Omega}^j_T)]{\int_{0}^{t} f(s,x) \D s}
		=& \sup_{(t,x) \in \overline{\Omega}^j_T} \AV{\int_{0}^{t} f(s,x) \D s} + \sup_{\substack{t_1,t_2,x\\t_1\neq t_2}} \frac{\AV{\int_{t_1}^{t_2} f(s,x) \D s}}{\AV{t_1-t_2}^{\gamma/2}}\\
		&+ \sup_{\substack{t,x_1,x_2\\x_1 \neq x_2}} \frac{\AV{\int_{0}^{t} f(s,x_1) - f(s,x_2) \D s}}{\AV{x_1-x_2}^\gamma}\\
		\leq& \bl \sup_{(t,x) \in \overline{\Omega}^j_T} \AV{f(t,x)} \br (T + T^{1-\gamma/2}) + \bl \sup_{\substack{t,x_1,x_2\\x_1 \neq x_2}} \frac{\AV{f(t,x_1) - f(t,x_2)}}{\AV{x_1-x_2}^\gamma} \br T
	\end{split}
\end{align*}
from which we conclude
\begin{align}\label{eq:intholder}
	\begin{split}
		\Norm[C^{0,\gamma}_\mathrm{par}(\overline{\Omega}^j_T)]{\int_{0}^{t} f(s,x) \D s}
		&\leq C T^{1-\gamma/2} \bl \sup_{(t,x) \in \overline{\Omega}^j_T} \AV{f(t,x)} + \sup_{\substack{t,x_1,x_2\\(t,x_1) \neq (t,x_2)}} \frac{\AV{f(t,x_1) - f(t,x_2)}}{\AV{x_1-x_2}^\gamma} \br \\
		&\leq C T^{1-\gamma/2} \Norm[C^{0,\gamma}_\mathrm{par}(\overline{\Omega}^j_T)]{f}
	\end{split}
\end{align}
for $f \in C^{0,\gamma}_\mathrm{par}(\overline{\Omega}^j_T)$.

\end{document}